\documentclass[pdflatex,sn-mathphys-num]{sn-jnl}
\usepackage{graphicx}%
\usepackage{multirow}%
\usepackage{amsmath,amssymb,amsfonts}%
\usepackage{amsthm}%
\usepackage{mathrsfs}%
\usepackage[title]{appendix}%
\usepackage{xcolor}%
\usepackage{textcomp}%
\usepackage{manyfoot}%
\usepackage{booktabs}%
\usepackage{algorithm}%
\usepackage{algorithmicx}%
\usepackage{algpseudocode}%
\usepackage{listings}%
\usepackage{float}%
\hypersetup{hypertexnames=false}

\numberwithin{equation}{section}

\theoremstyle{thmstyleone}
\newtheorem{theorem}{Theorem}[section]
\newtheorem{proposition}[theorem]{Proposition}
\newtheorem{lemma}[theorem]{Lemma}

\theoremstyle{thmstylethree}

\theoremstyle{thmstyletwo}
\newtheorem{remark}[theorem]{Remark}

\makeatletter
\renewcommand{\toclevel@subparagraph}{1}
\makeatother

\begin{document}

\title[Article Title]{Harmonic Map Heat Flow Coupled to a Green-Potential}

\author{Zhe Zhang}

\affil{
    \orgdiv{School of Arts and Sciences}, 
    \orgname{North China Institute of Aerospace Engineering}, 
    \orgaddress{
                            \street{Aimin East Road}, 
                            \city{Langfang}, 
                            \postcode{065099}, 
                            \state{Hebei}, 
                            \country{China}
                           }
      \textit{Correspondence: zhangzhe@nciae.edu.cn}
        }


\abstract{
We introduce a Green--potential heat flow for maps from a closed
Riemannian surface. Inspired by
Parker--Wolfson renormalization~\cite{ParkerWolfson1993,Parker1996}, the conformal factor is reconstructed
at each time from an energy-based source determined by the current
map. The resulting parabolic--elliptic system incorporates
renormalization into the evolution while leaving the Dirichlet energy
functional and its harmonic critical points unchanged. We establish
local well-posedness for $H^3$ initial data and a continuation criterion
in terms of energy concentration. For sufficiently large coupling, we exclude finite-time bubbling and obtain
global solutions. At infinite time, the energy identity provides only
weighted $L^2$ control of the tension field. Localizing to regular
regions of concentration charts yields harmonic profiles along
subsequences of prescribed times. We represent these profiles in the
flow's packet charts, organize them into a finite bubble tree, and
decompose the limiting energy into harmonic component energies and neck energies. A critically normalized radial example
illustrates how the modified flow can prevent finite-time blow-up and
produce harmonic bubbles with no-neck property.
}


\maketitle
\section{Introduction}

Energy concentration is a central feature of harmonic maps from surfaces.
Parker--Wolfson renormalization describes this concentration through
energy-determined centers and scales, leading to a tree of rescaled
components \cite{ParkerWolfson1993,Parker1996}.  This suggests an evolution
problem: can renormalization guide the choice of domain geometry while
the map evolves?

The conformal invariance of the 2-dimensional Dirichlet energy makes this
question natural.  Changing the conformal factor changes the metric and its gradient, but preserves both the energy functional and its
critical points.  One may therefore modify the evolution near
concentration without changing the harmonic maps sought by the flow.
We aim to incorporate energy-based renormalization into the equation
itself, with the conformal geometry determined by the flow.

The classical harmonic map heat flow admits global weak solutions with
possible concentration singularities \cite{Struwe1985}.  Other geometric
flows allow the domain metric to vary, including Teichm\"uller harmonic
map flow and harmonic Ricci flow
\cite{RupflinTopping2019,BuzanoRupflin2017}.
In Park's conformal heat flow, the conformal factor has its own evolution
equation and retains the preceding energy distribution
\cite{Park2022CHF}.  Here the geometry is instead reconstructed
instantaneously from an energy-based renormalization of the current map.
This couples the conformal response to the spatial configuration being
renormalized, without introducing a second initial datum or a time-history
term.

Let $(\Sigma,h)$ be a closed Riemannian surface and let $(N,g_N)$ be a closed $n$-dimensional Riemannian manifold.  We consider the parabolic--elliptic system
\begin{equation}
 \partial_t f=e^{-2\Phi(f)}\tau_h(f),\qquad
 -\Delta_h\Phi(f)=S_f,\qquad
 \int_\Sigma\Phi(f)\,dA_h=0.
 \label{eq:intro-gphf}
\end{equation}
The mean-zero source $S_f$, defined in Section~2, is a smooth
energy-based construction inspired by Parker--Wolfson renormalization;
it depends only on the current map.  We call
\eqref{eq:intro-gphf} the \emph{Green--potential heat flow} (GPHF).

Section~2 introduces the packet construction and the Green--potential heat flow. Section~3 establishes local well-posedness for $H^3$ initial maps and the continuation alternative. Section~4 excludes finite-time bubbling. The key observation is that the Green potential slows energy transport near a concentrating packet: the concentration would then have to persist back to a fixed earlier time, contradicting smoothness. Section~5 
studies long-time behavior of the global solutions and decomposes the limiting energy into harmonic bubble energies and possible neck energies. We conclude with a radial example illustrating the critical cylindrical regime, presenting the result while leaving its detailed analysis to a separate note.

\section{The Green--Potential Heat Flow}
\label{sec:green-potential-heat-flow}
 For $f\in H^1(\Sigma,N)$,
write
\begin{equation*}
 d\mu_f=e_h(f)\,dA_h,
 \qquad
 e_h(f)=\frac12|df|_h^2,
 \qquad
 E_h(f):=\mu_f(\Sigma).
\end{equation*}
Choose a geometric radius $0<r_* \leq\frac1{10}\operatorname{conv}(\Sigma,h)$
such that every ball $\overline{B_{2r}(p)}$, with
$p\in\Sigma$ and $0<r<r_*$, is strongly convex, namely there is a constant
$c_{\mathrm{cvx}}>0$ such that
\begin{equation}
 \nabla_x^2\!\left(\frac12d_h(x,y)^2\right)
 \geq c_{\mathrm{cvx}}h_x
 \label{eq:squared-distance-convexity}
\end{equation}
for all $x,y\in\overline{B_{2r}(p)}$.

Fix a nonincreasing cutoff
\[
 \chi\in C_c^\infty([0,\infty)),
 \qquad
 0\leq\chi\leq1,
 \qquad
 \chi\equiv1\ \text{on }[0,1],
 \qquad
 \operatorname{supp}\chi\subset[0,2],
\]
and set
\begin{equation*}
 \chi_{p,r}(y):=\chi\!\left(\frac{d_h(y,p)}r\right),
 \qquad
 M_f(p,r):=\int_\Sigma\chi_{p,r}\,d\mu_f.
\end{equation*}
Whenever $M_f(p,r)>0$, define the normalized localized energy
measure
\begin{equation*}
 d\eta_{f,p,r}(y)
 :=\frac{\chi_{p,r}(y)}{M_f(p,r)}\,d\mu_f(y).
\end{equation*}
The corresponding Fr\'echet functional is
\begin{equation*}
 \mathcal V_{f,p,r}(x)
 :=\int_\Sigma d_h(x,y)^2\,d\eta_{f,p,r}(y),
 \qquad x\in\overline{B_{2r}(p)}.
\end{equation*}
We define the packet barycenter and variance by
\begin{equation}\label{eq:packet-barycenter}
	b_f(p,r) :=\operatorname*{argmin}_{x\in\overline{B_{2r}(p)}}
       \mathcal V_{f,p,r}(x),\quad
 \sigma_f^2(p,r)
 :=\mathcal V_{f,p,r}\bigl(b_f(p,r)\bigr). 
\end{equation}
\begin{remark}
	This center-scale pair
$\bigl(b_f,\sigma_f\bigr)$  is inspired by the energy-based renormalization procedure of Parker and Wolfson, in which concentrating energy is recentered and rescaled using quantities determined by the energy measure \cite[\S4, pp.~83--84]{ParkerWolfson1993}.  \end{remark}
\begin{proposition}
\label{prop:packet-barycenter}
For $\forall$ $\varepsilon>0$, set
\begin{equation*}
 \mathcal A_\varepsilon
 :=\bigl\{(f,p,r)\in H^3(\Sigma,N)\times\Sigma\times(0,r_*):
          M_f(p,r)>\varepsilon\bigr\}.
\end{equation*}
For each $(f,p,r)\in\mathcal A_\varepsilon$, $\exists!$ 
$b_f(p,r)\in B_{2r}(p)$ which depends smoothly on 
 $(f,p,r)$.
\end{proposition}

\begin{proof}
By \eqref{eq:squared-distance-convexity}, for every
$x\in\overline{B_{2r}(p)}$ and $V\in T_x\Sigma$, we have
\begin{equation*}
	\nabla_x^2\mathcal V_{f,p,r}(V,V)\geq2c_{\mathrm{cvx}}|V|_h^2.
\end{equation*}
Thus $\mathcal V_{f,p,r}$ is uniformly strongly geodesically convex on
$\overline{B_{2r}(p)}$ and has a unique minimizer.  Suppose, for a
contradiction, that the minimizer $x$ lies on
$\partial B_{2r}(p)$.  Let
\[
 \nu_{\mathrm{in}}(x):=\frac{\exp_x^{-1}(p)}{2r}
\]
be the inward unit normal.  For $y\in B_{2r}(p)$, let $\gamma$ be the
minimizing geodesic from $x$ to $y$.  Since $B_{2r}(p)$ is strongly
convex, $\gamma((0,1])$ lies in the ball.  The strict convexity of
$s\mapsto\frac12d_h(p,\gamma(s))^2$, together with
$d_h(p,y)<d_h(p,x)=2r$, gives
\[
 \left\langle\exp_x^{-1}(y),\nu_{\mathrm{in}}(x)\right\rangle_h>0.
\]
Consequently,
\begin{equation*}
 D\mathcal V_{f,p,r}(x)[\nu_{\mathrm{in}}]
 =-2\int_\Sigma
   \left\langle\exp_x^{-1}(y),\nu_{\mathrm{in}}(x)\right\rangle_h
   \,d\eta_{f,p,r}(y)<0,\text{ contradicting minimality. }
\end{equation*}
For $x\in B_{2r}(p)$, denote
\begin{equation*}
 \mathcal B(f,p,r;x)
 :=\int_\Sigma\exp_x^{-1}(y)\,d\eta_{f,p,r}(y)
 =\frac{-1}{2}\nabla_x\mathcal V_{f,p,r}(x).
\end{equation*}
Thus $b_f(p,r)$ is the unique interior solution of
$\mathcal B(f,p,r;x)=0$. Smooth dependence follows from the implicit-function theorem and the fact that  
$D_x\mathcal B
 =-\frac12\nabla_x^2\mathcal V_{f,p,r}$ is invertible with inverse
bound depending only on $c_{\mathrm{cvx}}$. 
\end{proof}
The scale ratio $\rho_f(p,r) :=\frac{\sigma_f(p,r)}{r}$   is scale invariant and bounded above by 4.  Indeed, by Proposition~\ref{prop:packet-barycenter} the barycenter lies
in $B_{2r}(p)$, so $d_h(y,b_f(p,r))<4r$ on the support of the packet.

Fix a packet energy scale
$0<\varepsilon_{\mathrm{pkt}}\leq E_0$, where $E_0$ is a uniform upper
bound for the total energy.  Choose smooth nondecreasing functions
\begin{align*}
 \Gamma_1&:[0,\infty)\longrightarrow[0,1],
 &\Gamma_1(s)&=0 &&\text{for }s\leq\frac{\varepsilon_{\mathrm{pkt}}}{4},
 &\Gamma_1(s)&=1 &&\text{for }s\geq\frac{\varepsilon_{\mathrm{pkt}}}{2},
 \\
 \Gamma_2&:[0,\infty)\longrightarrow[0,1],
 &\Gamma_2(\rho)&=0 &&\text{for }\rho\leq\rho_-,
 &\Gamma_2(\rho)&=1 &&\text{for }\rho\geq\rho_+,
\end{align*}
where
\begin{equation*}
 \rho_-:=\frac1{16}\sqrt{\frac{\varepsilon_{\mathrm{pkt}}}{E_0}},
 \qquad
 \rho_+:=\frac14\sqrt{\frac{\varepsilon_{\mathrm{pkt}}}{E_0}}.
\end{equation*}
On the set where
$M_f(p,r)\leq\varepsilon_{\mathrm{pkt}}/4$, expressions involving
$\rho_f(p,r)$ are assigned weight zero.  Thus the combined packet
acceptance weight
\begin{equation*}
 W_f(p,r)
 :=\Gamma_1\bigl(M_f(p,r)\bigr)
   \Gamma_2\bigl(\rho_f(p,r)\bigr)
\end{equation*}
extends smoothly across the zero-mass region.

Define the total packet weight by
\begin{equation}
 \mathcal W(f)
 :=\int_0^{r_*}\int_\Sigma
 W_f(p,r)\,\frac{dA_h(p)}{r^2}\,\frac{dr}{r}.
 \label{eq:total-packet-weight}
\end{equation}
The factor $r^{-2}$ normalizes the area of an $r$-ball, while
$dr/r$ is invariant under rescaling.

Fix geometric constants $c_{\mathrm{vol}},C_{\mathrm{vol}}>0$ such that
\begin{equation}
 c_{\mathrm{vol}}r^2
 \leq \operatorname{Area}_h(B_r(p))
 \leq C_{\mathrm{vol}}r^2
 \qquad(p\in\Sigma,\ 0<r\leq4r_*).
 \label{eq:packet-volume-comparison}
\end{equation}

\begin{lemma}
\label{lem:total-packet-weight}
\begin{equation} \label{eq:total-packet-weight-bound}
\mathcal W(f)
 \leq 2^{17}C_{\mathrm{vol}} \left(\frac{E_0}{\varepsilon_{\mathrm{pkt}}}\right)^3. 
\end{equation} \end{lemma}

\begin{proof}
By the minimizing property \eqref{eq:packet-barycenter}, averaged over
the comparison point with respect to $\eta_{f,p,r}$,
\[
 M_f(p,r)^2\sigma_f(p,r)^2
 \leq\iint d_h(y,z)^2\chi_{p,r}(y)\chi_{p,r}(z)
                  \,d\mu_f(y)d\mu_f(z).
\]
On the active set $M_f>\frac{\varepsilon_{\mathrm{pkt}}}{4}$ and $\sigma_f>\rho_-r$.  Since $0\leq W_f\leq1$,
Tonelli's theorem gives
\[
 \mathcal W(f)\leq\frac1{{\frac{\varepsilon_{\mathrm{pkt}}}{4}}^2\rho_-^2}
 \iint_{y\ne z}d_h(y,z)^2
 \left[\int_0^{r_*}\int_\Sigma
 \chi_{p,r}(y)\chi_{p,r}(z)\,\frac{dA_h(p)\,dr}{r^5}\right]
 \,d\mu_f(y)d\mu_f(z).
\]
For $d=d_h(y,z)>0$, a nonzero product of cutoffs requires $r>d/4$,
and its centers lie in $B_{2r}(y)$.  By
\eqref{eq:packet-volume-comparison}, the bracket is at most
\[
 4C_{\mathrm{vol}}\int_{d/4}^{\infty}r^{-3}\,dr
 =\frac{32C_{\mathrm{vol}}}{d^2}.
\]
Substitution proves \eqref{eq:total-packet-weight-bound}.
\end{proof}
Let $H_\tau(x,y)$ denote the heat kernel of
$(\Sigma,h)$.   In normal coordinates, its leading small-time behavior is
\begin{equation*}
 H_\tau(x,y)
 \sim\frac1{4\pi\tau}
 \exp\!\left(-\frac{d_h(x,y)^2}{4\tau}\right).
\end{equation*}
Thus, for small $\tau$, $H_\tau(x,\cdot)$ is concentrated near $x$ on
a spatial scale comparable to $\sqrt\tau$.  For a center-scale pair
$(p,r)$ we take $\tau=r^2$ and define
\begin{equation}
 \nu_f(x)
 :=\int_{\Sigma\times\mathbb R_{>0}}
 H_{r^2}\bigl(x,b_f(p,r)\bigr)
 W_f(p,r)\,\frac{dA_h(p)\,dr}{r^3}.
 \label{eq:heat-reconstructed-pw-density}
\end{equation}
Here $W_f(p,r)$ is extended by zero for $r\geq r_*$.  With this
convention, \eqref{eq:heat-reconstructed-pw-density} is the heat-kernel
reconstruction of the parameter measure in
\eqref{eq:total-packet-weight}.

Fix a sufficiently large constant $\kappa>0$ and set
\begin{equation*}
 S_f(x):=\kappa P_h\nu_f(x), \text{ where $P_h$ is the mean-zero projection on $\Sigma$. }
\end{equation*}
We call $S_f$ the Parker--Wolfson source, or simply the PW source,
associated with $f$.

Given $f_0\in H^3(\Sigma,N)$, consider
\begin{equation}
 \left\{
 \begin{aligned}
  &\partial_t f=e^{-2\phi_t}\tau_h(f)=\tau_{g_t}(f),
  \quad g_t=e^{2\phi_t}h,
\\
  &-\Delta_h\phi_t
  =S_{f(t)},
\\
  &\int_\Sigma\phi_t\,dA_h
  =0,
 \end{aligned}
 \right.
 \label{eq:gphf-system}
\end{equation}
together with
\begin{equation*}
 f(0)=f_0.
\end{equation*}
This is a coupled parabolic--elliptic system. The map evolves by harmonic map heat flow with respect to the conformal metric $g_t$, while $\phi_t$ is the normalized Green potential of the instantaneous PW source $S_f$.  We call \eqref{eq:gphf-system} the \emph{Green--potential heat flow} (GPHF) with initial map $f_0$.


\begin{remark}
Let $\Phi(f)$ denote the normalized Green potential of $S_f$. On the smooth mapping space, the
GPHF is the negative gradient flow of $E_h$ for the inner product
\[
\mathsf G_f(X,Y)
:=\int_\Sigma e^{2\Phi(f)}\langle X,Y\rangle dA_h.
\]
If $E_h(f_0)\le\varepsilon_{\mathrm{pkt}}/4$, the PW source and its
normalized Green potential vanish. In this case, the GPHF agrees exactly with the ordinary harmonic map
heat flow.
\end{remark}

\section{Local Well--Posedness and Continuation}
\begin{theorem}
\label{thm:gphf-short-time}
Let $f_0\in H^3(\Sigma,N)$.  There is a number $T>0$ and a unique
solution pair $(f,\phi)$ of the GPHF on $[0,T]$ such that
\begin{equation*}
 f\in C\bigl([0,T];H^3(\Sigma,N)\bigr),
 \qquad
 \iota\circ f\in C^1\bigl([0,T];H^1(\Sigma,\mathbb R^L)\bigr)
\end{equation*}
for a fixed smooth isometric embedding $\iota:N\hookrightarrow\mathbb R^L$,
and
\begin{equation*}
 \phi\in C^1\bigl([0,T];C^\infty(\Sigma)\bigr)
\end{equation*}
in the graded Fr\'echet sense.  For every $0<\tau<T$,
\begin{equation*}
 f,\phi\in C^\infty\bigl(\Sigma\times[\tau,T]\bigr).
\end{equation*}
The equation in \eqref{eq:gphf-system} holds as an identity in
$H^1$ at every time.
Moreover,  one has the following energy identity
\begin{equation}\label{eq:short-time-energy-law}
 E_h(f(t_2))
 +\int_{t_1}^{t_2}\int_\Sigma
       a(f)|\tau_h(f)|^2\,dA_h\,dt
 =E_h(f(t_1))
\end{equation}
for  $a(f):=e^{-2\Phi(f)}$ and $0\leq t_1\leq t_2\leq T$.
\end{theorem}
Without loss of generality, fix a smooth isometric embedding 
\begin{equation*}
 \iota:(N,g_N)\hookrightarrow\mathbb R^L
\end{equation*}
and put $u_0:=\iota\circ f_0$.  We use the negative of the usual
second fundamental form,
\begin{equation*}
 A_\iota(X,Y):=-\operatorname{II}_\iota(X,Y).
\end{equation*}
The Gauss formula then gives
\begin{equation}
 d\iota\bigl(\tau_h(f)\bigr)
 =\Delta_hu+A_\iota(u)(du,du),
 \qquad u=\iota\circ f.
 \label{eq:extrinsic-tension-field}
\end{equation}

Let $\mathcal U$ be a tubular neighborhood of $\iota(N)$ and let
$\pi:\mathcal U\to\iota(N)$ be the nearest-point projection.  Choose
$\zeta\in C_c^\infty(\mathcal U)$ which equals one on a smaller tubular
neighborhood $\mathcal U_0\Subset\mathcal U$.  Since
\begin{equation*}
 D^2\pi_q(X,Y)=\operatorname{II}_\iota(X,Y),
 \qquad q\in\iota(N),\quad X,Y\in T_q\iota(N),
\end{equation*}
the formula
\begin{equation*}
 \widetilde A_\iota(z)(X,Y)
 :=-\zeta(z)D^2\pi_z(X,Y),
\end{equation*}
extended by zero outside $\mathcal U$, defines a smooth compactly
supported map
\[
 \widetilde A_\iota:
 \mathbb R^L\longrightarrow
 \operatorname{Bil}(\mathbb R^L\times\mathbb R^L;\mathbb R^L).
\]

Denote $
 \mathcal V_\delta^3
 :=\left\{u\in H^3(\Sigma,\mathbb R^L):
          \|u-u_0\|_{H^3}<\delta\right\}.
$
Since $H^3(\Sigma)\hookrightarrow C^{1,\alpha}(\Sigma)$, we may choose a small $\delta>0$ so that
\begin{equation*}
 u(\Sigma)\subset\mathcal U_0
 \quad\text{for every }u\in\mathcal V_\delta^3,
 \qquad
 C_\delta:=\sup_{u\in\mathcal V_\delta^3}
                  \|du\|_{L^\infty}<\infty.
\end{equation*}
Hence,  uniformly in $p\in\Sigma$,
\begin{equation}\label{eq:uniform-small-scale-packet-mass}
 M_u(p,r)
 =\int_\Sigma\chi_{p,r}\,d\mu_u
 \leq \frac{C_\delta^2}{2}
       \operatorname{Area}_h\bigl(B_{2r}(p)\bigr)
 \leq 2C_{\mathrm{vol}} C_\delta^2r^2.
\end{equation}
Choose $r_\delta\in(0,r_*)$ so that
$
 2C_{\mathrm{vol}} C_\delta^2r_\delta^2
 \leq\frac{\varepsilon_{\mathrm{pkt}}}{4}.
$
The mass cutoff from Section~2 then implies
$ W_u(p,r)=0\text{ for }u\in\mathcal V_\delta^3,
 \ p\in\Sigma,
 \ 0<r\leq r_\delta.
$

Choose a radius
$R>\|u_0\|_{H^3}+1$ and $\rho>0$ large enough that the $H^1$-open set
\begin{equation*}
 \mathcal V_\rho^1
 :=\left\{v\in H^1(\Sigma,\mathbb R^L):
          \|v-u_0\|_{H^1}<\rho\right\}
\end{equation*}
contains both $\mathcal V_\delta^3$ and the $H^3$-ball
$\{v:\|v\|_{H^3}<R\}$. 

Extend the packet construction to
$v\in\mathcal V_\rho^1$ by replacing $\mu_f$ with $\mu_v$ and
restricting the scale integration to $[r_\delta,r_*]$.
We denote the extended quantities by a tilde.
And clearly,
\begin{equation}
\widetilde S_\delta(u)=S_f
\qquad\text{whenever }u=\iota\circ f\in\mathcal V_\delta^3.
\label{eq:extended-source-agrees}
\end{equation}
\begin{proposition}
\label{prop:infinitely-smoothing-coefficient}
For every finite $\rho$, the map
\begin{equation*}
 \widetilde S_\delta:
 \mathcal V_\rho^1\longrightarrow C_0^\infty(\Sigma)
\end{equation*}
is $C^\infty$ in the graded Fr\'echet sense.  More precisely, for every
$m,\ell\geq0$,
\begin{equation}
 \left\|
 D^\ell\widetilde S_\delta(v)
 [z_1,\ldots,z_\ell]
 \right\|_{H^m}
 \leq C_{m,\ell}
       \prod_{j=1}^\ell\|z_j\|_{H^1}
 \label{eq:source-all-derivative-estimate}
\end{equation}
on every smaller $H^1$-ball whose closure is contained in
$\mathcal V_\rho^1$.  For $\ell=0$, the product is interpreted as $1$.
\end{proposition}

\begin{proof}
For fixed $(p,r)$, the mass depends polynomially on $dv$.  In
particular,
\begin{align*}
 D\widetilde M_v(p,r)[z]
 &=\int_\Sigma\chi_{p,r}\langle dv,dz\rangle\,dA_h,
 \\
 D^2\widetilde M_v(p,r)[z_1,z_2]
 &=\int_\Sigma\chi_{p,r}
            \langle dz_1,dz_2\rangle\,dA_h,
\end{align*}
and all higher derivatives vanish.  Hence
$v\mapsto\widetilde M_v(p,r)$ is smooth on $H^1$.

It remains to justify the normalized center and scale.  On the set
where the mass cutoff can be nonzero, define the unnormalized
barycenter vector field
\begin{equation*}
 \mathcal B(v,p,r;x)
 :=\int_\Sigma
   \exp_x^{-1}(y)\chi_{p,r}(y)\,d\mu_v(y).
\end{equation*}
The point $\widetilde b_v(p,r)$ is the unique zero of $\mathcal B$.
If
\begin{equation*}
 \mathcal E_{v,p,r}(x)
 :=\frac12\int_\Sigma
   d_h(x,y)^2\chi_{p,r}(y)\,d\mu_v(y),
\end{equation*}
then $\nabla_x\mathcal E_{v,p,r}=-\mathcal B$.  The convexity estimate
from Section~2 gives, for $\xi\in T_x\Sigma$,
\begin{equation}
	\bigl\langle-D_x\mathcal B[\xi],\xi\bigr\rangle_h
 =\operatorname{Hess}_x\mathcal E_{v,p,r}(\xi,\xi)
 \geq c_{\mathrm{cvx}}\widetilde M_v(p,r)|\xi|_h^2.
 \label{eq:barycenter-uniform-linearization}
\end{equation}

On the active set,
$\widetilde M_v(p,r)>\varepsilon_{\mathrm{pkt}}/4$; hence
$D_x\mathcal B$ is uniformly invertible there.  The Banach-space
implicit-function theorem shows that
\begin{equation*}
 v\longmapsto\widetilde b_v(p,r)
\end{equation*}
is smooth, with all derivatives bounded multilinearly by the $H^1$
norms of their arguments, locally uniformly in
$(p,r)\in\Sigma\times[r_\delta,r_*]$.

The mass cutoff bounds the denominator away from zero. To handle the square root, set 
\begin{equation*}
 q_v(p,r):=\frac{\widetilde\sigma_v(p,r)^2}{r^2},
 \qquad
 \widehat\Gamma_2(q):=\Gamma_2(\sqrt q).
\end{equation*}
Because $\Gamma_2$ vanishes identically on $[0,\rho_-]$,
$\widehat\Gamma_2$ extends as a smooth function at $q=0$.  Thus the
variance factor is the smooth composite
$\widehat\Gamma_2(q_v)$.

 Put
$m_*:=\varepsilon_{\mathrm{pkt}}/4$.  On
$\{\widetilde M_v(p,r)>m_*\}$, define, for $m\geq0$,
\begin{equation}
 \mathcal I_m(v,p,r;x)
 :=\Gamma_1(\widetilde M_v(p,r))
   \widehat\Gamma_2(q_v(p,r))
   \nabla_x^mH_{r^2}
      \bigl(x,\widetilde b_v(p,r)\bigr),
 \label{eq:weighted-heat-kernel-integrand}
\end{equation}
and set $\mathcal I_m=0$ when
$\widetilde M_v(p,r)\leq m_*$.  Since $\Gamma_1$ is smooth and
identically zero on $(-\infty,m_*]$, for every $j,N\geq0$,
\begin{equation}
 |\Gamma_1^{(j)}(s)|
 \leq C_{j,N}(s-m_*)_+^N
 \quad\text{for }s\text{ near }m_*.
 \label{eq:mass-cutoff-flatness}
\end{equation}
On the other hand, repeated implicit differentiation of the barycenter
equation produces only powers of
$(D_x\mathcal B)^{-1}$ and derivatives of $\mathcal B$.  The former
are uniformly bounded by \eqref{eq:barycenter-uniform-linearization}
because $\widetilde M_v\geq m_*$, and the latter are bounded on every
bounded $H^1$ set.  The same is true for all derivatives of $q_v$.
Consequently, every $v$-derivative of
\eqref{eq:weighted-heat-kernel-integrand} is a finite sum of bounded
factors multiplied by some derivative of $\Gamma_1$; by
\eqref{eq:mass-cutoff-flatness} it tends uniformly to zero as
$\widetilde M_v\downarrow m_*$.  Thus the zero extension of
$\mathcal I_m$ is $C^\infty$.  

Since $r\ge r_\delta>0$, uniform bounds for the heat kernel and
its derivatives, together with the preceding estimates, justify
differentiation under the integral to every order, proving the
asserted smoothness and \eqref{eq:source-all-derivative-estimate}.
\end{proof}
\begin{remark} Elliptic regularity and smooth composition imply that, for every
$m\ge0$,
\[
\widetilde\Phi_\delta:\mathcal V_\rho^1\to H_0^{m+2}(\Sigma),
\qquad
\widetilde a_\delta:\mathcal V_\rho^1\to H^m(\Sigma)
\]
are $C^\infty$, with derivative estimates analogous to
\eqref{eq:source-all-derivative-estimate} on the same smaller
$H^1$-balls.

Moreover, $0\le\widetilde W_v\le1$ and $r\ge r_\delta>0$ give a
uniform $L^2$ bound for $\widetilde S_\delta(v)$.
The elliptic estimate and $H^2(\Sigma)\hookrightarrow C^0(\Sigma)$
therefore yield
$|\widetilde\Phi_\delta(v)|\infty\le C\Phi$,
uniformly on $\mathcal V_\rho^1$. Consequently,
\begin{equation}\label{eq:uniform-ellipticity-extended-coefficient}
\lambda\le\widetilde a_\delta(v)(x)\le\Lambda,
\qquad
v\in\mathcal V_\rho^1,\quad x\in\Sigma,
\end{equation}
where $\lambda=e^{-2C_\Phi}$ and $\Lambda=e^{2C_\Phi}$.
\end{remark}

\subsection{Proof of the short-time well-posedness theorem}
\label{subsec:taylor-nonlocal-coefficient}

In local coordinates the extended ambient equation is
\begin{equation*}
 \partial_tu
 =\widetilde a_\delta(u)
  \left\{
  h^{jk}\partial_j\partial_ku+B(x,u,du)
  \right\},
\end{equation*}
where
\begin{equation*}
 B(x,u,du)
 =-h^{ij}\Gamma_{ij}^k\partial_ku
  +\widetilde A_\iota(u)(du,du).
\end{equation*}
By \eqref{eq:uniform-ellipticity-extended-coefficient}, its principal
symbol is a positive scalar multiple of $|\xi|_h^2I_{\mathbb R^L}$.

Taylor's Sobolev theorem for quasilinear parabolic systems treats
coefficients depending pointwise on a finite spatial jet of $u$; see
\cite[Chapter~15, Sections~7--8, especially Proposition~8.2]
{TaylorPDEIII}.  Our coefficient
$u\mapsto\widetilde a_\delta(u)$ is not of this literal form, because
it depends on the map on all of $\Sigma$.  Proposition
\ref{prop:infinitely-smoothing-coefficient}, however, gives stronger
estimates than the local Moser estimates used in that construction: an
$H^1$ change of the input produces an $H^m$ change of the coefficient
for every $m$.  The next lemma records the precise variant needed here.

\begin{lemma}
\label{lem:taylor-smoothing-nonlocal}
Let $U\subset H^1(\Sigma,\mathbb R^L)$ be open, and suppose
\begin{equation*}
 a:U\longrightarrow C^\infty(\Sigma)
\end{equation*}
is a positive scalar map satisfying, on bounded subsets of $U$,
\begin{align}
 \lambda&\leq a(v)(x)\leq\Lambda,
 \notag\\
 \left\|D^\ell a(v)[z_1,\ldots,z_\ell]\right\|_{H^m}
 &\leq C_{m,\ell}
       \prod_{j=1}^\ell\|z_j\|_{H^1}
 \qquad(m,\ell\geq0).
 \label{eq:abstract-coefficient-tame-estimate}
\end{align}
Let $A:\mathbb R^L\to
\operatorname{Bil}(\mathbb R^L\times\mathbb R^L;\mathbb R^L)$ be
smooth with bounded derivatives.  If $u_0\in H^3$ and an $H^3$-ball
containing $u_0$ has closure contained in $U$, then
\begin{equation*}
 \partial_tu=a(u)\{\Delta_hu+A(u)(du,du)\},
 \qquad u(0)=u_0,
\end{equation*}
has a unique solution on some nonmaximal interval $[0,T]$ with
\begin{equation*}
 u\in C([0,T];H^3)\cap C^1([0,T];H^1).
\end{equation*}
It is smooth for positive time.
\end{lemma}
\begin{proof}
\noindent\emph{Step 1: Friedrichs regularization.}
Write $\Lambda_h=(I-\Delta_h)^{1/2}$.
Choose $\psi\in C_c^\infty(\mathbb R)$ with $0\le\psi\le1$ and
$\psi=1$ near the origin, and set
$J_\eta=\psi(\eta\Lambda_h)$ for $0<\eta\le1$.
Then $J_\eta$ is self-adjoint, commutes with $\Lambda_h$ and
$\Delta_h$, and is contractive on every $H^s$. Moreover,
\[
 \|J_\eta v\|_{H^{s+k}}\le C_k\eta^{-k}\|v\|_{H^s},
 \qquad
 \|(I-J_\eta)v\|_{H^{s-1}}\le C\eta\|v\|_{H^s},
\]
and $J_\eta v\to v$ in $H^s$ for each $v\in H^s$.
Write
\[
 F(v)=a(v)\{\Delta_hv+A(v)(dv,dv)\},
 \qquad Q(v)=a(v)A(v)(dv,dv),
\]
and consider
\begin{equation}
 \partial_tu_\eta=J_\eta F(J_\eta u_\eta),
 \qquad u_\eta(0)=u_0.
 \label{eq:regularized-equation}
\end{equation}
By \eqref{eq:abstract-coefficient-tame-estimate} and Sobolev
multiplication, $F:H^3\cap U\to H^1$ is locally Lipschitz.
Since $J_\eta$ preserves the $H^3$-ball of radius $R$, the
regularized vector field is locally Lipschitz there for each fixed
$\eta$. The Banach-space ODE theorem gives a local solution $u_\eta$.

\medskip
\noindent\emph{Step 2: uniform estimates.}
Put $w_\eta=J_\eta u_\eta$. Self-adjointness and commutation give
\[
 \frac12\frac d{dt}\|u_\eta\|_{H^3}^2
 =\langle\Lambda_h^3F(w_\eta),\Lambda_h^3w_\eta\rangle.
\]
As long as $\|u_\eta\|_{H^3}<R$, contractivity gives
$\|w_\eta\|_{H^3}<R$, so both maps lie in $U$.
The coefficient bounds and commutator estimates yield
\[
 \|[\Lambda_h^3,a(w_\eta)]\Delta_hw_\eta\|_{L^2}
 \le C_R\|w_\eta\|_{H^4},
 \qquad \|Q(w_\eta)\|_{H^2}\le C_R.
\]
Integration by parts and uniform ellipticity give, for some $c>0$,
\[
 \langle a(w_\eta)\Delta_h\Lambda_h^3w_\eta,
                 \Lambda_h^3w_\eta\rangle
 \le -c\|w_\eta\|_{H^4}^2+C_R\|w_\eta\|_{H^3}^2.
\]
Using
\[
 \langle\Lambda_h^3Q(w_\eta),\Lambda_h^3w_\eta\rangle
 =\langle\Lambda_h^2Q(w_\eta),\Lambda_h^4w_\eta\rangle
\]
and absorbing the commutator and quadratic terms by Young's inequality,
we obtain
\begin{equation}
 \frac d{dt}\|u_\eta\|_{H^3}^2
 +c_R\|J_\eta u_\eta\|_{H^4}^2
 \le C_R(1+\|u_\eta\|_{H^3}^2).
 \label{eq:uniform-h3-estimate}
\end{equation}
The equation also gives $\|\partial_tu_\eta\|_{H^1}\le C_R$.

Choose $T>0$ such that
\[
 (\|u_0\|_{H^3}^2+1)e^{C_RT}-1<R^2.
\]
Gronwall's inequality excludes exit from the $H^3$-ball before $T$.
Contractivity keeps $J_\eta u_\eta$ in the same ball, so the
coefficient remains defined. ODE continuation and
\eqref{eq:uniform-h3-estimate} therefore give
\begin{equation}
 \sup_{0<\eta\le1}
 \left(
 \|u_\eta\|_{L^\infty(0,T;H^3)}
 +\|J_\eta u_\eta\|_{L^2(0,T;H^4)}
 +\|\partial_tu_\eta\|_{L^\infty(0,T;H^1)}
 \right)<\infty.
 \label{eq:uniform-space-estimates}
\end{equation}
\medskip
\noindent\emph{Step 3: passage to the limit.}
Interpolation gives
\[
 \|u_\eta(t)-u_\eta(s)\|_{H^2}
 \le C_R|t-s|^{1/2}.
\]
Rellich compactness and Arzel\`a--Ascoli yield, after passing to a
subsequence,
\[
 u_\eta\to u\quad\text{in }C([0,T];H^2),
 \qquad
 u_\eta\stackrel{*}{\rightharpoonup}u
 \quad\text{in }L^\infty(0,T;H^3).
\]
The mollifier estimate gives $J_\eta u_\eta\to u$ in
$C([0,T];H^2)$. Its weak $L^2H^4$ limit is therefore also $u$, so
\eqref{eq:uniform-space-estimates} implies $u\in L^2(0,T;H^4)$.
By \eqref{eq:abstract-coefficient-tame-estimate},
\[
 a(J_\eta u_\eta)\to a(u)
 \quad\text{in }C([0,T];H^m)\quad(m\ge0),
\]
and hence $F(J_\eta u_\eta)\to F(u)$ in $C([0,T];L^2)$.
Strong convergence of $J_\eta$ on $L^2$, uniform on compact subsets,
allows passage to the limit in \eqref{eq:regularized-equation}:
\begin{equation}
 u(t)=u_0+\int_0^tF(u(s))\,ds
 \quad\text{in }L^2.
 \label{eq:limit-integral-equation}
\end{equation}

The bounds $u\in L^\infty H^3\cap L^2H^4$ imply
$F(u)\in L^2H^2$, and thus $u_t\in L^2H^2$.
The Hilbert-triple continuity theorem, with
$H^4\subset H^3\subset H^2$, gives
\[
 u\in C([0,T];H^3),\qquad u(0)=u_0.
\]
Since $F:H^3\cap U\to H^1$ is continuous,
\eqref{eq:limit-integral-equation} then gives
\[
 u\in C^1([0,T];H^1),\qquad u_t=F(u)\quad\text{in }H^1.
\]

\medskip
\noindent\emph{Step 4: uniqueness and positive-time regularity.}
Let $u,v$ be two solutions with the same initial datum, and set
$z=u-v$. On a common bounded convex $H^1$-neighborhood in $U$, with
both $H^3$ norms bounded,
\[
 z_t=a(u)\Delta_hz+[a(u)-a(v)]\Delta_hv+Q(u)-Q(v),
\]
and
\[
 \|a(u)-a(v)\|_{L^\infty}
 +\|Q(u)-Q(v)\|_{L^2}
 \le C_R\|z\|_{H^1}.
\]
Pairing with $z$, integrating by parts, and applying Young's
inequality gives
\[
 \frac d{dt}\|z\|_{L^2}^2
 +\frac\lambda2\|dz\|_{L^2}^2
 \le C_R\|z\|_{L^2}^2.
\]
Gronwall's inequality proves local uniqueness, and repetition proves
uniqueness on the common interval of existence.

For each integer $k\ge4$, the coefficient bounds and tame estimate
\[
 \|Q(u)\|_{H^{k-1}}\le C_{k,R}(1+\|u\|_{H^k})
\]
give, whenever the indicated regularity holds,
\begin{equation}
 \frac d{dt}\|u\|_{H^k}^2
 +c_k\|u\|_{H^{k+1}}^2
 \le C_{k,R}(1+\|u\|_{H^k}^2).
 \label{eq:higher-order-energy-estimate}
\end{equation}
For $H^k$ initial data with $H^3$ norm below $R$, repeat
Steps~1--3 using both \eqref{eq:uniform-h3-estimate} and
\eqref{eq:higher-order-energy-estimate}. This gives local solutions
in $C_tH^k\cap L_t^2H^{k+1}$ while retaining the same $H^3$ bound.
Fix $0<\tau<T$ and choose $t_4\in(0,\tau)$ with $u(t_4)\in H^4$.
Restart the construction at $t_4$ and identify the resulting solution
with $u$ by uniqueness. Estimate
\eqref{eq:higher-order-energy-estimate}, together with the fixed
distance of $u([0,T])$ from $U^c$ in $H^1$, permits continuation to
$T$, giving
\[
 u\in C([t_4,T];H^4)\cap L^2(t_4,T;H^5).
\]
Inductively choose $t_k\in(t_{k-1},\tau)$ with $u(t_k)\in H^k$
and repeat. Thus $u\in C([\tau,T];H^k)$ for every $k$.
The equation and the smooth dependence of $a$ then give all time
derivatives, proving positive-time smoothness.
\end{proof}
\begin{proof}[Proof of Theorem~\ref{thm:gphf-short-time}]
Apply Lemma~\ref{lem:taylor-smoothing-nonlocal} with
$a=\widetilde a_\delta$, $A=\widetilde A_\iota$, and
$U=\mathcal V_\rho^1$.  Proposition~\ref{prop:infinitely-smoothing-coefficient}
verifies its hypotheses and gives an ambient solution
\begin{equation*}
 \begin{gathered}
 u\in C([0,T];H^3(\Sigma,\mathbb R^L))
       \cap C^1([0,T];H^1(\Sigma,\mathbb R^L)),\\
 \partial_tu=\widetilde a_\delta(u)
 \{\Delta_hu+\widetilde A_\iota(u)(du,du)\}.
 \end{gathered}
\end{equation*}
After decreasing $T$, continuity gives
$u(t)\in\mathcal V_\delta^3$ and $u(t)(\Sigma)\subset\mathcal U_0$.

To verify target preservation, set $R(z)=z-\pi(z)$ and
$q(z)=|R(z)|^2/2$.  The nearest-point projection satisfies
\[
 Dq_z[X]=\langle R(z),X\rangle,\qquad
 D^2q_z[X,X]=|(I-D\pi_z)X|^2
             -\langle R(z),D^2\pi_z[X,X]\rangle.
\]
Since $\widetilde A_\iota=-D^2\pi$ on $\mathcal U_0$, the chain rule yields
\begin{equation*}
 (\partial_t-\widetilde a_\delta(u)\Delta_h)q(u)
 =-\widetilde a_\delta(u)|(I-D\pi_u)du|_h^2\leq0.
\end{equation*}
Apply the maximum principle on $[\epsilon,T]$ and let
$\epsilon\downarrow0$.  As $q(u(\epsilon))\to q(u_0)=0$ uniformly,
we obtain $u(t)\in\iota(N)$.

Thus $f=\iota^{-1}\circ u$ solves the original equation by
\eqref{eq:extended-source-agrees} and
\eqref{eq:extrinsic-tension-field}.  Set
$\phi(t)=\widetilde\Phi_\delta(u(t))$.  The chain rule in every $H^m$ gives
\[
 \partial_t\phi
 =D\widetilde\Phi_\delta(u)[\partial_tu],
\]
and hence the asserted $C^1$ regularity with values in $C^\infty$.
Positive-time smoothness and uniqueness follow from
Lemma~\ref{lem:taylor-smoothing-nonlocal}, using the same ambient
extension on successive common neighborhoods.
Finally, the first variation gives
$dE_h(f)/dt=-\int_\Sigma a(f)|\tau_h(f)|^2\,dA_h$;
The established regularity justifies this identity up to $t=0$.
Integration proves \eqref{eq:short-time-energy-law}.
\end{proof}

\subsection{The maximal solution and the concentration alternative}
\label{subsec:continuation-concentration-alternative}

Theorem~\ref{thm:gphf-short-time} and uniqueness give a maximal strong
solution on $[0,T_{\max})$.
By \eqref{eq:uniform-small-scale-packet-mass}, the local construction has a uniform existence time on bounded $H^3$ initial data.  And the corresponding ambient extension has uniform ellipticity.  If $\|\iota\circ f(t)\|_{H^3}$ remained bounded at a
finite $T_{\max}$, restarting sufficiently close to $T_{\max}$ and
using uniqueness would extend the solution beyond it. We now show that finite-time failure of continuation leads to energy concentration.
\begin{lemma}
\label{lem:scale-normalized-small-energy}
There is a constant $0<\varepsilon_{\mathrm{reg}}
\leq\varepsilon_{\mathrm{pkt}}/4$, depending only on the fixed
geometric and packet data, $\kappa$, and $E_0$, with the following
property. Let $f$ be a smooth GPHF solution on $(s,T)$, where
$T<\infty$, with
$E_h(f(t))\leq E_0$. If, for some $0<r<r_*$,
\begin{equation}
 \sup_{s<t<T}\sup_{x\in\Sigma}
 E_h\bigl(f(t);B_{4r}(x)\bigr)
 \leq\varepsilon_{\mathrm{reg}},
 \label{eq:gphf-uniform-small-energy}
\end{equation}
then, for every $s<s'<T$ and every $k\geq1$,
\begin{equation*}
 \sup_{s'<t<T}
 \|\nabla^k(\iota\circ f(t))\|_{L^\infty(\Sigma)}<\infty.
\end{equation*}
These bounds may additionally depend on $k$, $r$, $s'-s$, and $T-s$;
the threshold $\varepsilon_{\mathrm{reg}}$ is independent of them.
\end{lemma}

\begin{proof}
Write $u=\iota\circ f$ and $a(x,t)=a(f(t))(x)$.
Since $\operatorname{supp}\chi_{p,q}\subset B_{2q}(p)$,
\eqref{eq:gphf-uniform-small-energy} gives
\[
 M_{f(t)}(p,q)\leq\varepsilon_{\mathrm{reg}}
 \leq\frac{\varepsilon_{\mathrm{pkt}}}{4},
 \qquad 0<q\leq r.
\]
Consequently $W_{f(t)}(p,q)=0$ on these scales.

Let $G_h$ be the mean-zero Green kernel of $-\Delta_h$ and set
\begin{equation}
 K_q(x,z):=\int_\Sigma G_h(x,y)H_{q^2}(y,z)\,dA_h(y)
 =\int_{q^2}^{\infty}
 \left(H_\tau(x,z)-\frac1{\operatorname{Area}_h(\Sigma)}\right)d\tau.
 \label{eq:heat-regularized-green-kernel}
\end{equation}
The heat-kernel derivative bounds at small times and exponential
decay of the mean-zero heat kernel at large times give
\[
 \|K_q(\cdot,z)\|_{L^\infty}
 \leq C(1+|\log q|),
\]
\begin{equation}
 \|\nabla_x^mK_q(\cdot,z)\|_{L^\infty}
 \leq C_mq^{-m},\qquad m\geq1,
 \label{eq:green-heat-kernel-derivative-bound}
\end{equation}
uniformly for $0<q<r_*$ and $z\in\Sigma$.
Introduce the parameter measure
\begin{equation}
 d\Lambda_t(p,q)
 :=W_{f(t)}(p,q)\frac{dA_h(p)\,dq}{q^3}.
 \label{eq:packet-parameter-measure}
\end{equation}
By the total-weight bound \eqref{eq:total-packet-weight-bound} and
the energy law \eqref{eq:short-time-energy-law},
$\Lambda_t(\Sigma\times(0,r_*))\leq C_W$.
Green reconstruction reads
\begin{equation}
 \Phi(f(t))(x)
 =\kappa\int_{\Sigma\times(0,r_*)}
 K_q\bigl(x,b_{f(t)}(p,q)\bigr)\,d\Lambda_t(p,q).
 \label{eq:potential-as-parameter-integral}
\end{equation}
Here the integral is supported on $q\geq r$, so
\begin{equation}
 \|\Phi(f(t))\|_{L^\infty}
 \leq C\kappa C_W(1+|\log r|),\qquad
 r^m\|\nabla^m\Phi(f(t))\|_{L^\infty}
 \leq\kappa C_WC_m.
 \label{eq:gphf-small-energy-potential-bounds}
\end{equation}
In particular, for fixed $r$ there is $\lambda_r>0$ such that
$\lambda_r\leq a\leq\lambda_r^{-1}$ throughout $\Sigma\times(s,T)$.
For $x\in B_{4r}(x_0)$, the first-derivative bound gives
$|\Phi(x,t)-\Phi(x_0,t)|\leq C$, independently of $r$.
The chain rule therefore yields
\[
 c_0^{-1}\leq\frac{a(x,t)}{a(x_0,t)}\leq c_0,
 \qquad
 r^m\left\|\nabla^m\frac{a(\cdot,t)}{a(x_0,t)}
        \right\|_{L^\infty(B_{4r}(x_0))}\leq c_m,
\]
with $c_0,c_m$ independent of $r$, $x_0$, and $t$.

Fix $x_0$ and use the coordinates $x=\exp_{x_0}(ry)$ and the time
variable
\[
 \theta(t):=r^{-2}\int_s^t a(x_0,q)\,dq.
\]
The rescaled map $v(y,\theta(t))=u(\exp_{x_0}(ry),t)$ satisfies on
$B_4$ the equation
\[
 v_\theta=b\{\Delta_gv+A_\iota(v)(dv,dv)\},
 \qquad b(y,\theta(t))=\frac{a(\exp_{x_0}(ry),t)}{a(x_0,t)},
\]
where $g$ is the rescaled domain metric. The ellipticity and all spatial
coefficient bounds are uniform in $r$ and $x_0$.
For $e=|dv|_g^2$, the Bochner formula gives
\[
 (\partial_\theta-b\Delta_g)e\leq C(e+e^2).
\]
The term involving $db$ is bounded by
$2|db|\,|\tau_g(v)|\,|dv|$ and is absorbed into the negative Hessian
term, up to $Ce$.

We recall the point-selection argument to make the independence of
the energy threshold explicit. On a backward cylinder
$Q_\ell=B_\ell\times[\theta_1-\ell^2,\theta_1]$ contained in the
rescaled domain, with $0<\ell\leq1$, put
\[
 d(y,\theta)
 :=\min\{\ell-|y|,\sqrt{\theta-\theta_1+\ell^2}\}.
\]
If $\max_{Q_\ell}d^2e$ is sufficiently large, choose a maximizing
point $(y_*,\theta_*)$ and write $e_*=e(y_*,\theta_*)$ and
$d_*=d(y_*,\theta_*)$. On its backward cylinder of radius $d_*/2$,
$e\leq4e_*$. Rescaling space by $e_*^{1/2}$ and time by $e_*$ gives a
unit backward cylinder on which
\[
 \widehat e(0,0)=1,\qquad 0\leq\widehat e\leq4,\qquad
 (\partial_\vartheta-\widehat b\Delta_{\widehat g})\widehat e
 \leq C\widehat e.
\]
The parabolic mean-value inequality and the conformal invariance of
two-dimensional energy imply
\[
 1\leq C_{\mathrm{mv}}\iint_{Q_1}\widehat e
       \,dA_{\widehat g}\,d\vartheta
 \leq2C_{\mathrm{mv}}\varepsilon_{\mathrm{reg}}.
\]
Choosing
$\varepsilon_{\mathrm{reg}}
\leq\min\{\varepsilon_{\mathrm{pkt}}/4,(4C_{\mathrm{mv}})^{-1}\}$
excludes this possibility. The mean-value estimate requires no time
derivative of $b$, since
$b\Delta_g e=\operatorname{div}_g(b\nabla e)-\langle db,de\rangle_g$.
Moreover,
$\theta(s')\geq\lambda_r(s'-s)/r^2>0$.
Thus cylinders with
$\ell\leq\min\{1,\sqrt{\theta(s')/2}\}$ give a uniform bound for
$|du|$ on $\Sigma\times(s',T)$.

For completeness, higher derivatives follow using only spatial
coefficient bounds. First obtain the gradient bound on
$\Sigma\times(s_1,T)$, where $s_1=(s+s')/2$.
By \eqref{eq:gphf-small-energy-potential-bounds}, all spatial
derivatives of $a$ are bounded for this fixed $r$.
Spatial differentiation, the Sobolev product estimates, and
integration by parts give, for $Y_k(t)=\|u(t)\|_{H^k}^2$,
\[
 \frac d{dt}Y_k+c_rY_{k+1}\leq C_{k,r}(1+Y_k),\qquad k\geq1,
\]
where the constants also depend on the established gradient bound.
The nonlinear term is controlled by
\[
 \|aA_\iota(u)(du,du)\|_{H^{k-1}}
 \leq C_{k,r}(1+\|u\|_{H^k}).
\]
Since $Y_1$ is bounded by the energy and compactness of $N$,
integration first bounds $\int Y_2$ on an interior time interval.
Choose a slice on which $Y_2$ is bounded and apply Gronwall's
inequality. Repeating this procedure on successive subintervals of
$(s_1,s')$ bounds every $Y_k$ on $(s',T)$.
Sobolev embedding proves the assertion.
\end{proof}

For $p\in\Sigma$ and $0<T\leq T_{\max}$ with $T<\infty$, define
the concentration mass
\[
 \mathcal K(p,T)
 :=\lim_{r\downarrow0}\limsup_{t\uparrow T}
 E_h\bigl(f(t);B_r(p)\bigr).
\]
\begin{theorem}[Continuation--concentration alternative]
\label{thm:continuation-concentration-alternative}
Let $f$ be the maximal strong solution with $E_h(f_0)\leq E_0$, and
set $\varepsilon_{\mathrm{cont}}:=\varepsilon_{\mathrm{reg}}/4$,
where $\varepsilon_{\mathrm{reg}}$ is the constant in
Lemma~\ref{lem:scale-normalized-small-energy}.
If $T_{\max}<\infty$, then
\[
 \mathcal K(p,T_{\max})\geq\varepsilon_{\mathrm{cont}}
 \quad\text{for some }p\in\Sigma.
\]
\end{theorem}

\begin{proof}
Suppose instead that $\mathcal K(p,T_{\max})
<\varepsilon_{\mathrm{cont}}$ for every $p$.
Choose $R_p>0$ such that
\[
 \limsup_{t\uparrow T_{\max}}
 E_h\bigl(f(t);B_{R_p}(p)\bigr)
 <\varepsilon_{\mathrm{cont}},
\]
and take a finite subcover $\{B_{R_{p_i}/4}(p_i)\}_{i=1}^J$ of
$\Sigma$. Choose $0<r<r_*$ with
$r\leq\min_iR_{p_i}/16$. Every ball $B_{4r}(x)$ is then contained
in some $B_{R_{p_i}}(p_i)$. By finiteness of the cover, there is one
$s<T_{\max}$ such that
\[
 \sup_{s<t<T_{\max}}\sup_{x\in\Sigma}
 E_h\bigl(f(t);B_{4r}(x)\bigr)
 <2\varepsilon_{\mathrm{cont}}
 <\varepsilon_{\mathrm{reg}}.
\]
Lemma~\ref{lem:scale-normalized-small-energy} now gives a uniform
$H^3$ bound on every shorter interval $(s',T_{\max})$, contradicting
the continuation criterion. 
\end{proof}

\section{Exclusion of Finite--Time Bubbling}
\label{sec:finite-time-exclusion}
In this section, we exclude finite-time energy concentration and
establish global existence for the GPHF. A harmonic profile extracted from a presumed concentration provides a block of accepted PW packets, the persistence of which would force energy concentration at a fixed earlier time, contradicting the regularity.
\begin{lemma}
\label{lem:universal-sphere-packet}
For fixed $N,E_0$ and $\chi$, there are constants
\begin{equation*}
 m_{\mathrm{sph}},\rho_{\mathrm{sph}},c_x,c_s>0
\end{equation*}
such that every nonconstant finite-energy harmonic map
$\omega:\mathbb R^2\to N$ with $E(\omega)\leq E_0$ admits a pair
$(z_\omega,r_\omega)$ satisfying
\begin{equation*}
 M_\omega(z,r)\geq m_{\mathrm{sph}},\qquad
 \rho_\omega(z,r)\geq\rho_{\mathrm{sph}}
\end{equation*}
throughout the relative block
\begin{equation}
 |z-z_\omega|\leq c_xr_\omega,\qquad
 |\log(r/r_\omega)|\leq c_s.
 \label{eq:universal-sphere-parameter-block}
\end{equation}
The packet variables are Euclidean and use the given plane coordinate;
only translation and dilation are required.
\end{lemma}

\begin{proof}
By removable singularities \cite{SacksUhlenbeck1981}, $\omega$ extends
smoothly over infinity, so $|D\omega(z)|=O(|z|^{-2})$ there.
Hence $L=\max_{\mathbb R^2}|D\omega|>0$ is attained at some $z_\omega$.
The harmonic map $v(\xi)=\omega(z_\omega+L^{-1}\xi)$ satisfies
$|Dv|\leq1$ and $|Dv(0)|=1$.  Interior elliptic estimates give a
uniform $C^{0,\alpha}$ bound for $Dv$ on $B_1$, depending only on $N$.
Choose $\delta>0$ so that $|Dv|\geq1/2$ on $B_{4\delta}$ and put
$r_\omega=\delta/L$.  The energy density of $\omega$ is then between
$L^2/8$ and $L^2/2$ on $B_{4r_\omega}(z_\omega)$.

Take $c_x=1/4$ and $c_s=\log(3/2)$.  Every cutoff support in
\eqref{eq:universal-sphere-parameter-block} lies strictly inside
$B_{4r_\omega}(z_\omega)$, and $r\geq r_\omega/2$.
Since $\chi_{z,r}=1$ on $B_r(z)$,
\[
 M_\omega(z,r)\geq\frac{\pi L^2r^2}{8}
 \geq\frac{\pi\delta^2}{32}.
\]
For every $b\in\mathbb R^2$,
$\int_{B_r(z)}|y-b|^2\,dy\geq\pi r^4/2$.
Using $M_\omega\leq E_0$ therefore gives
\[
 \rho_\omega(z,r)^2
 \geq\frac{\pi L^2r^2}{16E_0}
 \geq\frac{\pi\delta^2}{64E_0}.
\]
Thus one may take $m_{\mathrm{sph}}=\pi\delta^2/32$ and
$\rho_{\mathrm{sph}}=(\pi\delta^2/(64E_0))^{1/2}$.
\end{proof}

Choose the free packet threshold so that
\begin{equation}
 0<\varepsilon_{\mathrm{pkt}}
 <\min\{m_{\mathrm{sph}},E_0\rho_{\mathrm{sph}}^2,E_0\}.
 \label{eq:finite-time-packet-compatibility}
\end{equation}

\begin{theorem}[Finite-time regularity]
\label{thm:finite-time-regularity}
Set
\begin{equation}
 a_0:=2c_sc_{\mathrm{vol}}c_x^2e^{-2c_s},
 \qquad \gamma:=\frac{\kappa a_0}{\pi}.
 \label{eq:finite-time-gamma}
\end{equation}
If the fixed packet data satisfy
\eqref{eq:finite-time-packet-compatibility} and
$
 \gamma\geq2,
$
then   $T_{\max}=\infty$.
\end{theorem}
The energy identity~\eqref{eq:short-time-energy-law} gives
$
 \mathscr D(s,t):=E(s)-E(t)
 =\int_s^t\int_\Sigma a_q|\tau_h(f(q))|^2\,dA_h\,dq.
$
For a time-independent smooth test $\psi$, put
$I_\psi(t)=\int_\Sigma\psi\,d\mu_t$.  Differentiation and integration
by parts yield
\begin{equation}
 \frac{d}{dt}I_\psi(t)
 =-\int_\Sigma\psi a_t|\tau_h(f)|^2\,dA_h
  -\int_\Sigma
  a_t\langle df,\nabla\psi\otimes\tau_h(f)\rangle_h\,dA_h.
 \label{eq:signed-localized-energy-identity}
\end{equation}
Hence, weighted Cauchy--Schwarz gives
\begin{equation}
	 |I_\psi(t)-I_\psi(s)|
 \leq\|\psi\|_\infty\mathscr D(s,t)
 +\mathscr D(s,t)^{1/2}
 \left(\int_s^t\int_\Sigma
 a_q|df|^2|\nabla\psi|^2\,dA_h\,dq\right)^{1/2}.
 \label{eq:universal-localized-transport}
\end{equation}

\subsection{Harmonic packet detection}
\label{subsec:first-concentration-scales}
\label{subsec:almost-harmonic-packet-detection}
A first-crossing argument selects times and scales with positive
localized energy and vanishing rescaled tension.  The resulting
harmonic profile then provides a whole block of accepted PW packets.
\begin{lemma}
\label{lem:terminal-harmonic-profile}
Suppose $T=T_{\max}<\infty$ and
$\mathcal K(p,T)\geq\varepsilon_{\mathrm{cont}}$ at some $p\in\Sigma$.
Then there are times $v_j\uparrow T$, centers $p_j\in\Sigma$, scales
$\lambda_j\downarrow0$, and a nonconstant finite-energy harmonic map
$\omega:\mathbb R^2\to N$ such that
\begin{equation}
 f\bigl(v_j,\exp_{p_j}(\lambda_jz)\bigr)
 \longrightarrow\omega(z)
 \quad\text{strongly in }W^{1,2}_{\mathrm{loc}}(\mathbb R^2).
 \label{eq:terminal-harmonic-profile-convergence}
\end{equation}
In particular, $E(\omega)\leq E_0$ and $\omega$ extends to a harmonic
sphere.
\end{lemma}

\begin{proof}
By hypothesis,
\begin{equation*}
 \mathcal K(p,T)\geq\varepsilon_{\mathrm{cont}}.
\end{equation*}
Since $E$ is nonincreasing, the
limit $E(T^-):=\lim_{t\uparrow T}E(t)$ exists.  Set
\begin{equation*}
 \omega_E(s):=E(s)-E(T^-).
\end{equation*}
Then $\omega_E(s)\downarrow0$ as $s\uparrow T$.

Define the first-crossing mass
\begin{equation*}
 m_{\mathrm{fc}}
 :=\frac18\min\{\varepsilon_{\mathrm{cont}},
                 \varepsilon_{\mathrm{pkt}}\}.
\end{equation*}
In particular,
$0<m_{\mathrm{fc}}<\varepsilon_{\mathrm{pkt}}/4$ and
$m_{\mathrm{fc}}<\mathcal K(p,T)$.

Write $C_K=C_1$ for the first-derivative constant in
\eqref{eq:green-heat-kernel-derivative-bound}; $C_W$ is fixed in
\eqref{eq:total-packet-weight-bound}.
Choose any sequence $L_j\uparrow\infty$, with $L_j\geq4$, and put
\begin{equation}
 H_j:=\exp(4\kappa C_KC_WL_j).
 \label{eq:first-crossing-comparison-factor}
\end{equation}
Choose $s_j\uparrow T$ sufficiently rapidly that
\begin{equation*}
 H_j\omega_E(s_j)\longrightarrow0.
\end{equation*}
Since $f(s_j)$ is smooth, there are $R_j\downarrow0$ such that
\begin{equation}
 L_jR_j\longrightarrow0,
 \qquad
 F_j(s_j):=\sup_{x\in\Sigma,\ 0\leq r\leq R_j}
 M_{f(s_j)}(x,r)<\frac{m_{\mathrm{fc}}}{4}.
 \label{eq:first-crossing-starting-scale}
\end{equation}
Here $M_f(x,0):=0$.  The estimate is uniform in $x$ because
$M_{f(s_j)}(x,r)\leq C\|df(s_j)\|_\infty^2r^2$.

For every fixed $R_j$, the definition of $\mathcal K(p,T)$ gives a
later time at which
$E_h(f(t);B_{R_j}(p))>m_{\mathrm{fc}}$.  Since
$\chi_{p,R_j}=1$ on $B_{R_j}(p)$, one has $F_j(t)>m_{\mathrm{fc}}$ at
such a time.  On a compact time interval contained in $[0,T)$, the map
$(t,x,r)\mapsto M_{f(t)}(x,r)$ is continuous, including at $r=0$.
Consequently $F_j$ is continuous and attains its supremum.  Let
$\bar t_j>s_j$ be the first time for which
\begin{equation*}
 F_j(\bar t_j)=m_{\mathrm{fc}},
\end{equation*}
and choose $(\bar x_j,\varrho_j)$, with
$0<\varrho_j\leq R_j$, such that
\begin{equation}
 M_{f(\bar t_j)}(\bar x_j,\varrho_j)=m_{\mathrm{fc}}.
 \label{eq:first-crossing-packet}
\end{equation}
For all $q\in[s_j,\bar t_j]$, all $x\in\Sigma$, and all
$0<r\leq R_j$,
\begin{equation*}
 M_{f(q)}(x,r)\leq m_{\mathrm{fc}}
 <\frac{\varepsilon_{\mathrm{pkt}}}{4}.
\end{equation*}
The flat part of $\Gamma_1$ therefore gives
\begin{equation*}
 W_{f(q)}(x,r)=0
 \qquad
 (0<r\leq R_j,\ s_j\leq q\leq\bar t_j).
\end{equation*}
This is the first-crossing input: all finer source scales are inactive
on the entire interval, not merely at its terminal time.

The Green representation
\eqref{eq:potential-as-parameter-integral}, the total-weight bound,
and \eqref{eq:green-heat-kernel-derivative-bound} imply
\begin{equation}
 \|\nabla\phi_q\|_{L^\infty(\Sigma)}
 \leq\frac{\kappa C_KC_W}{R_j}
 \qquad(s_j\leq q\leq\bar t_j).
 \label{eq:first-crossing-potential-gradient}
\end{equation}
Let
\begin{equation*}
 \mathcal O_j:=B_{L_j\varrho_j}(\bar x_j),
 \qquad
 \alpha_j(q):=\sup_{\mathcal O_j}a_q.
\end{equation*}
Since $\log a_q=-2\phi_q$, equations
\eqref{eq:first-crossing-comparison-factor} and
\eqref{eq:first-crossing-potential-gradient} give
\begin{equation}
 H_j^{-1}\alpha_j(q)\leq a_q(y)\leq\alpha_j(q)
 \qquad(y\in\mathcal O_j).
 \label{eq:first-crossing-mobility-comparison}
\end{equation}

For $J=[u,v]\subset[s_j,\bar t_j]$, define its packet-scale effective
length by
\begin{equation*}
 \ell_j(J):=\frac1{\varrho_j^2}\int_u^v\alpha_j(q)\,dq.
\end{equation*}
Apply \eqref{eq:universal-localized-transport} to the fixed test
$\chi_{\bar x_j,\varrho_j}$.  Its derivative support lies in
$\mathcal O_j$, $|\nabla\chi_{\bar x_j,\varrho_j}|\leq C/\varrho_j$,
and $\int_\Sigma|df(q)|^2\,dA_h\leq2E_0$.  Hence
\begin{align}
 &\left|M_{f(v)}(\bar x_j,\varrho_j)
       -M_{f(u)}(\bar x_j,\varrho_j)\right|\notag\\
 &\hspace{18mm}\leq
 \mathscr D(u,v)
 +C\sqrt{E_0\mathscr D(u,v)\ell_j(J)}.
 \label{eq:first-crossing-effective-transport}
\end{align}

If $\ell_j([s_j,\bar t_j])\leq1$ along a subsequence, then
\eqref{eq:first-crossing-starting-scale},
\eqref{eq:first-crossing-packet}, and
\eqref{eq:first-crossing-effective-transport} would give
\[
 \frac{3m_{\mathrm{fc}}}{4}
 \leq\omega_E(s_j)+C\sqrt{E_0\omega_E(s_j)}=o(1),
\]
a contradiction.  Thus $\ell_j([s_j,\bar t_j])>1$ for all large $j$.
Because $\alpha_j(q)>0$, there is a unique
$q_j\in(s_j,\bar t_j)$ satisfying
\begin{equation*}
 \ell_j([q_j,\bar t_j])=1.
\end{equation*}
Equation \eqref{eq:first-crossing-effective-transport}, applied
between any $q\in[q_j,\bar t_j]$ and $\bar t_j$, now yields
\begin{equation}
 M_{f(q)}(\bar x_j,\varrho_j)
 =m_{\mathrm{fc}}+o(1)
 \geq\frac{m_{\mathrm{fc}}}{2}.
 \label{eq:first-crossing-mass-survival}
\end{equation}

Using the lower inequality in
\eqref{eq:first-crossing-mobility-comparison}, we obtain
\begin{align*}
 \omega_E(s_j)
 &\geq\int_{q_j}^{\bar t_j}\int_{\mathcal O_j}
             a_q|\tau_h(f(q))|^2\,dA_h\,dq\\
 &\geq H_j^{-1}\int_{q_j}^{\bar t_j}\alpha_j(q)
       \int_{\mathcal O_j}|\tau_h(f(q))|^2\,dA_h\,dq.
\end{align*}
The measure
$\varrho_j^{-2}\alpha_j(q)dq$ has total mass one on
$[q_j,\bar t_j]$.  Averaging therefore gives a time
$v_j\in[q_j,\bar t_j]$ such that
\begin{equation}
 \varrho_j^2
 \int_{B_{L_j\varrho_j}(\bar x_j)}
 |\tau_h(f(v_j))|^2\,dA_h
 \leq H_j\omega_E(s_j)\longrightarrow0.
 \label{eq:first-crossing-small-rescaled-tension}
\end{equation}
Together with \eqref{eq:first-crossing-mass-survival}, this proves
\begin{equation}
 M_{f(v_j)}(\bar x_j,\varrho_j)
 \geq\frac{m_{\mathrm{fc}}}{2},
 \qquad
 L_j\varrho_j\longrightarrow0,
 \qquad v_j\longrightarrow T.
 \label{eq:first-crossing-selected-data}
\end{equation}
After passing to a subsequence we may suppose $v_j\uparrow T$.

In exponential coordinates define
\begin{equation*}
 u_j(z):=f\bigl(v_j,\exp_{\bar x_j}(\varrho_jz)\bigr),
 \qquad z\in B_{L_j}(0),
\end{equation*}
and let $g_j$ be the corresponding rescaled domain metric.  Then
$g_j\to g_{\mathbb R^2}$ smoothly on compact sets.  Equations
\eqref{eq:first-crossing-small-rescaled-tension} and
\eqref{eq:first-crossing-selected-data} give
\begin{equation}
 \sup_jE_{g_j}(u_j;B_{L_j})\leq E_0,
 \qquad
 E_{g_j}(u_j;B_2)\geq\frac{m_{\mathrm{fc}}}{2},
 \qquad
 \|\tau_{g_j}(u_j)\|_{L^2(B_{L_j},g_j)}\longrightarrow0.
 \label{eq:almost-harmonic-input-data}
\end{equation}

Apply local approximate-harmonic compactness and the energy identity
to \eqref{eq:almost-harmonic-input-data}
\cite{DingTian1995,QingTian1997}; see also
\cite[Theorem~1.1 and Section~2]{WangWeiZhang2017}.
The interior estimates are uniform under the smooth convergence
of $g_j$, and diagonal extraction gives a harmonic base component
on $\mathbb R^2$, with finitely many concentration points.
If the base is constant, the positive energy retained in $B_2$
produces a nonconstant sphere; otherwise the base itself is
nonconstant.

Finite-energy removability and the harmonic-sphere energy gap
\cite{SacksUhlenbeck1981} make the energetic descendant tree finite.
Choose a terminal nonconstant component $\omega$.  It has no
remaining concentration points, so convergence in its chart is
strong in $W^{1,2}_{\mathrm{loc}}(\mathbb R^2)$.
Returning to physical coordinates gives
\eqref{eq:terminal-harmonic-profile-convergence} with
$\lambda_j=O(\varrho_j)\to0$ and $E(\omega)\le E_0$.
\end{proof}

Apply Lemma~\ref{lem:universal-sphere-packet} to the profile in
Lemma~\ref{lem:terminal-harmonic-profile} and transfer its robust block
through \eqref{eq:terminal-harmonic-profile-convergence}.  After
incorporating the fixed translation and dilation of the profile pair,
we obtain physical center-radius pairs $(x_j,r_j)$ with $r_j\downarrow0$.
For
\begin{equation}
 \mathcal B_j
 :=\left\{(x,r):
 d_h(x,x_j)\leq c_xr_j,
 \ |\log(r/r_j)|\leq c_s\right\},
 \label{eq:robust-physical-parameter-block}
\end{equation}
strong local $W^{1,2}$ convergence gives, uniformly on
$\mathcal B_j$,
\begin{equation*}
 M_{f(v_j)}(x,r)\geq\frac{m_{\mathrm{sph}}}{2},
 \qquad
 \rho_{f(v_j)}(x,r)\geq\frac{\rho_{\mathrm{sph}}}{2}
\end{equation*}
for all large $j$.  The uniformity follows because on a compact
relative parameter block the cutoff mass and its first and second
moments converge uniformly; uniform strict convexity of the rescaled
Karcher functionals then gives convergence of their minimizers.

Define the accepted set at time $t$ by
\begin{equation*}
 \operatorname{Acc}(t)
 :=\left\{(x,r):
 M_{f(t)}(x,r)>\frac{\varepsilon_{\mathrm{pkt}}}{2},
 \quad \rho_{f(t)}(x,r)>\rho_+\right\}.
\end{equation*}
Every pair in $\operatorname{Acc}(t)$ has $W_{f(t)}(x,r)=1$.
Fix constants 
$\delta_M,\delta_\rho>0$ such that
\begin{equation}
 M_{f(v_j)}(x,r)
 \geq\frac{\varepsilon_{\mathrm{pkt}}}{2}+4\delta_M,
 \qquad
 \rho_{f(v_j)}(x,r)\geq\rho_++4\delta_\rho
 \label{eq:robust-block-strict-margins}
\end{equation}
throughout $\mathcal B_j$.  Thus
$\mathcal B_j\subset\operatorname{Acc}(v_j)$ with strict, uniform
margins.

For later use, choose one fixed $L_\Omega>0$ so that the supports of
all packet cutoffs and their first derivatives for pairs in
$\mathcal B_j$ lie in
$
 \Omega_j:=B_{L_\Omega r_j}(x_j).
$

\subsection{The Green barrier, no-flux, and backward persistence}
\label{subsec:green-potential-barrier}

Recall the parameter measure $\Lambda_t$ from
\eqref{eq:packet-parameter-measure}.
By Lemma~\ref{lem:total-packet-weight},
\begin{equation}
 \Lambda_t(\Sigma\times(0,r_*))\leq C_W
 \label{eq:finite-time-total-source-bound}
\end{equation}
uniformly in $t$.  In terms of the regularized Green kernel from
\eqref{eq:heat-regularized-green-kernel}, the reconstruction formula is
\begin{equation}
 \phi_t(y)
 =\kappa\int_{\Sigma\times(0,r_*)}
 K_r\bigl(y,b_{f(t)}(p,r)\bigr)\,d\Lambda_t(p,r).
 \label{eq:finite-time-green-parameter-reconstruction}
\end{equation}

We shall use the standard estimates
\begin{align}
 &G_h(y,z)\geq\frac1{2\pi}\log\frac1{d_h(y,z)}-C_G,
 G_h(y,z)\geq-C_G,
 \label{eq:green-kernel-lower-bounds}\\
 &\int_\Sigma d_h(u,z)H_{r^2}(u,z)\,dA_h(u)\leq C_Hr.
 \label{eq:heat-kernel-first-moment}
\end{align}

\begin{proposition}
\label{prop:accepted-block-green-barrier}
There is a constant $C_{\mathrm{wall}}$, independent of $j$ and $t$,
such that, for all sufficiently large $j$,
\begin{equation}
 \mathcal B_j\subset\operatorname{Acc}(t)
 \quad\Longrightarrow\quad
 \sup_{y\in\Omega_j}a_t(y)
 \leq C_{\mathrm{wall}}r_j^\gamma.
 \label{eq:accepted-block-green-barrier}
\end{equation}
\end{proposition}

\begin{proof}
If $\mathcal B_j\subset\operatorname{Acc}(t)$, then $W_{f(t)}=1$ on
the entire block.  Since
$r_j e^{-c_s}\leq r\leq r_je^{c_s}$ there, equations
\eqref{eq:packet-volume-comparison} and
\eqref{eq:robust-physical-parameter-block} give
\begin{align}
 \Lambda_t(\mathcal B_j)
 &=\int_{\mathcal B_j}\frac{dA_h(p)}{r^2}\frac{dr}{r}\notag\\
 &\geq2c_se^{-2c_s}r_j^{-2}
       \operatorname{Area}_h(B_{c_xr_j}(x_j))
 \geq a_0.
 \label{eq:accepted-block-source-mass}
\end{align}

If $y\in\Omega_j$ and $(p,r)\in\mathcal B_j$, then the barycenter
location in Proposition~\ref{prop:packet-barycenter} gives
\begin{equation*}
 d_h\bigl(b_{f(t)}(p,r),x_j\bigr)\leq Cr_j.
\end{equation*}
Convolving the first inequality in
\eqref{eq:green-kernel-lower-bounds} against
$H_{r^2}(\,\cdot\,,b_{f(t)}(p,r))$, using
\eqref{eq:heat-kernel-first-moment}, and applying Jensen's inequality
to the convex function $-\log$ yields
\begin{equation}
 K_r\bigl(y,b_{f(t)}(p,r)\bigr)
 \geq\frac1{2\pi}\log\frac1{r_j}-C
 \qquad(y\in\Omega_j,\ (p,r)\in\mathcal B_j).
 \label{eq:regularized-green-block-lower-bound}
\end{equation}
The second inequality in \eqref{eq:green-kernel-lower-bounds} gives
$K_r\geq-C_G$ everywhere.  Split
\eqref{eq:finite-time-green-parameter-reconstruction} into the block
and its complement.  Using
\eqref{eq:finite-time-total-source-bound},
\eqref{eq:accepted-block-source-mass}, and
\eqref{eq:regularized-green-block-lower-bound}, we find
\begin{equation*}
 \phi_t(y)
 \geq\frac{\kappa a_0}{2\pi}\log\frac1{r_j}-C
 \qquad(y\in\Omega_j).
\end{equation*}
Since $a_t=e^{-2\phi_t}$, exponentiation and
\eqref{eq:finite-time-gamma} prove
\eqref{eq:accepted-block-green-barrier}.
\end{proof}

The barrier is conditional on acceptance.  We next control transport
on such an interval and then establish backward persistence.

\begin{lemma}
\label{lem:accepted-block-transport}
Suppose
\begin{equation*}
 \mathcal B_j\subset\operatorname{Acc}(q)
 \qquad(t_1\leq q\leq t_2).
\end{equation*}
Let $\psi_{j,x,r}$ be any time-independent test function, indexed by
$(x,r)\in\mathcal B_j$, satisfying
\begin{equation}
 \|\psi_{j,x,r}\|_\infty\leq C,
 \qquad
 \|\nabla\psi_{j,x,r}\|_\infty\leq Cr_j^{-1},
 \qquad
 \operatorname{supp}\nabla\psi_{j,x,r}\subset\Omega_j.
 \label{eq:accepted-test-function-bounds}
\end{equation}
Then, uniformly for $(x,r)\in\mathcal B_j$,
\begin{equation}
	 |I_{\psi_{j,x,r}}(t_2)-I_{\psi_{j,x,r}}(t_1)|
 \leq C\mathscr D(t_1,t_2)
 +C\sqrt{E_0(t_2-t_1)\mathscr D(t_1,t_2)}
       r_j^{(\gamma-2)/2}.
 \label{eq:accepted-block-transport-estimate}
\end{equation}

Moreover, the signed flux satisfies
\begin{equation}
	\left|\int_{t_1}^{t_2}\int_\Sigma
 a_q\langle df,\nabla\psi_{j,x,r}\otimes\tau_h(f)\rangle_h
 \,dA_h\,dq\right|
 \leq
 C\sqrt{E_0(t_2-t_1)\mathscr D(t_1,t_2)}
 r_j^{(\gamma-2)/2}.
 \label{eq:accepted-block-no-flux}
\end{equation}
 
\end{lemma}

\begin{proof}
The accepted-interval hypothesis permits us to apply
Proposition~\ref{prop:accepted-block-green-barrier} at every time in
$[t_1,t_2]$.  On the derivative support of the test function,
\begin{equation*}
 a_q|\nabla\psi_{j,x,r}|^2\leq Cr_j^{\gamma-2}.
\end{equation*}
Since $\int_\Sigma|df(q)|^2\,dA_h\leq2E_0$, the second factor in
\eqref{eq:universal-localized-transport} is bounded by
$C(E_0(t_2-t_1)r_j^{\gamma-2})^{1/2}$.  This proves
\eqref{eq:accepted-block-transport-estimate}.  Applying the same
weighted Cauchy--Schwarz estimate directly to the second term in
\eqref{eq:signed-localized-energy-identity} gives
\eqref{eq:accepted-block-no-flux}.
\end{proof}

For $s<T$ define
\begin{equation*}
 \varepsilon_{\mathrm{tr}}(s)
 :=C\omega_E(s)+C\sqrt{E_0T\omega_E(s)}.
\end{equation*}
If $\gamma\geq2$ and $r_j<1$, then on every accepted interval
$[t_1,t_2]\subset[s,T)$,
\begin{equation}
 |I_{\psi_{j,x,r}}(t_2)-I_{\psi_{j,x,r}}(t_1)|
 \leq\varepsilon_{\mathrm{tr}}(s),
 \qquad
 \varepsilon_{\mathrm{tr}}(s)\longrightarrow0
 \quad(s\uparrow T).
 \label{eq:critical-terminal-transport}
\end{equation}
At the critical exponent $\gamma=2$, no positive power of $r_j$ is
available; the vanishing dissipation tail supplies exactly the required
smallness.

To preserve packet acceptance, we control the variance ratio as well
as the localized mass using \eqref{eq:critical-terminal-transport}.
Fix $(x,r)\in\mathcal B_j$ and an accepted interval
$[t_1,t_2]\subset[s,T)$.  Write
$M_i=M_{f(t_i)}(x,r)$ and $\rho_i=\rho_{f(t_i)}(x,r)$.
For each fixed $b\in\overline B_{2r}(x)$, the test
\[
 \psi_{b,x,r}(y):=\frac{\chi_{x,r}(y)}{r^2}d_h(y,b)^2
\]
satisfies \eqref{eq:accepted-test-function-bounds} uniformly in $b$:
on the cutoff support, $d_h(y,b)\le4r$, so
$\|\psi_{b,x,r}\|_\infty\le16$ and
$\|\nabla\psi_{b,x,r}\|_\infty\le Cr^{-1}\le Cr_j^{-1}$.
The gradient is supported in $\Omega_j$.
Define
\[
 J_i(b):=\int_\Sigma\psi_{b,x,r}\,d\mu_{t_i}.
\]
Applying \eqref{eq:critical-terminal-transport} to $\chi_{x,r}$
and $\psi_{b,x,r}$ gives
\[
 |M_1-M_2|+
 \sup_{b\in\overline B_{2r}(x)}|J_1(b)-J_2(b)|
 \le C\varepsilon_{\mathrm{tr}}(s).
\]
Here $b$ is held fixed during each time comparison.
Since $0\le J_i(b)\le16M_i$ and
$M_i\ge\varepsilon_{\mathrm{pkt}}/2$, normalization yields
\[
 \sup_{b\in\overline B_{2r}(x)}
 \left|\frac{J_1(b)}{M_1}-\frac{J_2(b)}{M_2}\right|
 \le C\varepsilon_{\mathrm{tr}}(s).
\]
By definition,
$\rho_i^2=\min_{b\in\overline B_{2r}(x)}J_i(b)/M_i$.
Comparing the minimum values and using $\rho_i\ge\rho_+$, we obtain
\[
 |\rho_1-\rho_2|
 =\frac{|\rho_1^2-\rho_2^2|}{\rho_1+\rho_2}
 \le C\varepsilon_{\mathrm{tr}}(s).
\]
Consequently,
\begin{align}
 &\sup_{(x,r)\in\mathcal B_j}
 \bigl(
 |M_{f(t_1)}(x,r)-M_{f(t_2)}(x,r)|
 +|\rho_{f(t_1)}(x,r)-\rho_{f(t_2)}(x,r)|
 \bigr)\notag\\
 &\hspace{26mm}\le C\varepsilon_{\mathrm{tr}}(s)
 \label{eq:accepted-packet-moment-transport}
\end{align}
on every accepted interval $[t_1,t_2]\subset[s,T)$.
Fix the constant $C$ in this estimate once and for all; it depends
only on the fixed packet geometry, $\varepsilon_{\mathrm{pkt}}$
and $\rho_+$, not on $j$, the pair $(x,r)$ or the times.

\begin{proposition}
\label{prop:uniform-backward-persistence}
There is one time $s_*<T$, independent of $j$, such that for every
sufficiently large $j$,
\begin{equation*}
 \mathcal B_j\subset\operatorname{Acc}(t)
 \qquad(s_*\leq t\leq v_j),
\end{equation*}
and, more precisely,
\begin{equation}
 M_{f(t)}(x,r)
 \geq\frac{\varepsilon_{\mathrm{pkt}}}{2}+\delta_M,
 \qquad
 \rho_{f(t)}(x,r)\geq\rho_++\delta_\rho
 \label{eq:backward-persistence-margins}
\end{equation}
for all $t\in[s_*,v_j]$ and $(x,r)\in\mathcal B_j$.
\end{proposition}

\begin{proof}
Choose $s_*<T$ sufficiently close to $T$ that
\begin{equation}
 C\varepsilon_{\mathrm{tr}}(s_*)
 <\min\{\delta_M,\delta_\rho\}.
 \label{eq:backward-persistence-tail-choice}
\end{equation}
Increase $j$ so that $v_j>s_*$.  At $v_j$ the compact block has the
fourfold margins in \eqref{eq:robust-block-strict-margins}.  Ordinary
continuity of the packet mass and variance in time and in the
center-radius variables therefore gives a nontrivial interval ending at
$v_j$ on which the whole block is accepted.

Set
\begin{equation*}
 \mathfrak J_j
 :=\{q\in[s_*,v_j]:
       \mathcal B_j\subset\operatorname{Acc}(t)
       \text{ for every }t\in[q,v_j]\},
 \qquad
 \underline s_j:=\inf\mathfrak J_j.
\end{equation*}
The preceding observation gives $\mathfrak J_j\ne\varnothing$.

Suppose $\underline s_j>s_*$.  For every
$t\in(\underline s_j,v_j]$, the block is accepted throughout
$[t,v_j]$.  Hence the conditional Green barrier and transport estimate
are valid on that interval.  Applying
\eqref{eq:accepted-packet-moment-transport} between $t$ and $v_j$, and
then letting $t\downarrow\underline s_j$, gives, uniformly on
$\mathcal B_j$,
\begin{equation*}
 M_{f(\underline s_j)}(x,r)
 \geq\frac{\varepsilon_{\mathrm{pkt}}}{2}+2\delta_M,
 \qquad
 \rho_{f(\underline s_j)}(x,r)
 \geq\rho_++2\delta_\rho.
\end{equation*}
Here we used \eqref{eq:robust-block-strict-margins} and
\eqref{eq:backward-persistence-tail-choice}.  The whole compact block
therefore lies strictly inside the accepted set at
$\underline s_j$.  Ordinary continuity extends the acceptance interval
slightly to the left, contradicting the definition of
$\underline s_j$.  Thus $\underline s_j=s_*$.

Choose $q_n\downarrow s_*$ with $q_n\in\mathfrak J_j$.  Apply the same
transport estimate on $[q_n,v_j]$ and let $n\to\infty$.  Time
continuity gives the strict margins at $s_*$ and hence
$s_*\in\mathfrak J_j$.  The same estimates at every intermediate time
give \eqref{eq:backward-persistence-margins} on the full closed interval.
\end{proof}

Each block is detected only at its selected time $v_j$.
The proposition supplies the separate persistence argument that carries
it to the common earlier time $s_*$.

\begin{proof}[Proof of Theorem~\ref{thm:finite-time-regularity}]
If $T_{\max}<\infty$, the continuation--concentration alternative
and Lemma~\ref{lem:terminal-harmonic-profile} give the selected profiles
above.  Proposition~\ref{prop:uniform-backward-persistence} implies
\[
 M_{f(s_*)}(x_j,r_j)
 \geq\frac{\varepsilon_{\mathrm{pkt}}}{2}+\delta_M.
\]
But $f(s_*)$ is smooth at the fixed time $s_*<T_{\max}$, and
\eqref{eq:packet-volume-comparison} gives
\[
 M_{f(s_*)}(x_j,r_j)
 \leq2C_{\mathrm{vol}}\|df(s_*)\|_\infty^2r_j^2\longrightarrow0.
\]
This contradiction proves $T_{\max}=\infty$.
\end{proof}
\section{Bubble-Tree Analysis at Infinite Time}
\label{sec:long-time-ridge-tree}

The energy identity provides only weighted $L^2$ control of the Dirichlet energy tension field. By
localizing the energy estimates to regular regions of the PW packet
charts, we extract harmonic profiles and organize them into a finite
bubble tree along a subsequence. We then establish an energy identity  including both bubble energies  and the
energy remaining in the necks.

\subsection{Localized transport and the PW mobility}
\label{subsec:lt-local-transport}
The argument selects an almost-harmonic map on a compact
region at a comparison time and transports its strong $H^1$ limit
back to the prescribed time.  Local profile transport is also used
in \cite[Lemma~2.10]{JendrejLawrieSchlag2025}; here the instantaneous
PW source supplies the mobility comparison needed for this step.

We use the parameter measure of Section~2,
\begin{equation*}
 d\Lambda_t(p,q)=W_{f(t)}(p,q)\frac{dA_h(p)\,dq}{q^3},\qquad
 \Lambda_t\bigl(\Sigma\times(0,r_*)\bigr)\le C_W.
\end{equation*}
Let $m_{\rm src}=\varepsilon_{\rm pkt}/4$ denote the mass-onset
threshold.  For every active packet, $W_{f(t)}(p,q)>0$ implies
\begin{equation}
 M_{f(t)}(p,q)>m_{\rm src},\quad
 d_h(p,b_t(p,q))\le2q,\quad
 \operatorname{supp}\chi_{p,q}\subset B_{2q}(p).
 \label{eq:lt-new-active-packet}
\end{equation}
The localized transport argument uses only these consequences of
the cutoffs and the total weight bound.  The PW representation in
the final subsection also uses the acceptance margins fixed in
Section~4; no cutoff is changed here.

\begin{lemma}
\label{lem:lt-new-green-comparison}
Set $D_\Sigma=1+\operatorname{diam}_h\Sigma+r_*$.  The
heat-regularized Green kernel satisfies
\begin{equation}
 K_q(x,b):=\int_\Sigma G_h(x,y)H_{q^2}(y,b)\,dA_h(y)
 =\frac1{2\pi}\log\frac{D_\Sigma}{q+d_h(x,b)}+O(1),
 \label{eq:lt-new-log-kernel}
\end{equation}
uniformly for $0<q<r_*$ and $x,b\in\Sigma$.  Suppose $X,Y\subset\Sigma$
satisfy $d_h(x,y)\le C_1r$ for $x\in X$, $y\in Y$, and every active
packet satisfies
\begin{equation}
 q+d_h(x,b_t(p,q))\ge c_1r\qquad(x\in X).
 \label{eq:lt-new-source-exclusion}
\end{equation}
Then, with constants independent of $r$ and $t$,
\begin{equation}
 \inf_X a(t,\cdot)\ge c\sup_Y a(t,\cdot),\qquad
 \inf_X a(t,\cdot)\ge
 c_0\min\left\{1,\left(\frac{c_1r}{D_\Sigma}\right)^{\kappa C_W/\pi}\right\}>0.
 \label{eq:lt-new-mobility-comparison}
\end{equation}
The constants may depend on $c_1,C_1$, the fixed geometry, $\kappa$
and $C_W$.
\end{lemma}

\begin{proof}
The semigroup identity is
\[
 K_q(x,b)=\int_{q^2}^{\infty}
 \left(H_s(x,b)-|\Sigma|_h^{-1}\right)ds.
\]
For a fixed small $s_0>0$ the part $s\ge s_0$ is uniformly bounded.
Writing $d=d_h(x,b)$, the singular part of the heat parametrix is
\[
 \frac1{4\pi}\int_{q^2}^{s_0}e^{-d^2/(4s)}\,\frac{ds}{s}
 =\frac1{2\pi}\log^+\frac{\sqrt{s_0}}{\max\{q,d\}}+O(1).
\]
The integrated parametrix remainder is bounded: its leading
coefficient differs from one by $O(d^2)$ and its next term is
integrable in $s$.  Splitting the displayed integral at $s=d^2$
proves the asserted logarithmic estimate.  Off the diagonal and for
$q^2\ge s_0$ the kernel is bounded.  Replacing the truncated
logarithm by that in \eqref{eq:lt-new-log-kernel} changes only a bounded
term.

For an active packet, put $d_x=q+d_h(x,b_t(p,q))$ and define $d_y$
similarly.  The hypotheses give $d_y\le d_x+C_1r\le(1+C_1/c_1)d_x$.
Thus $K_q(x,b)-K_q(y,b)\le C$.  Integrating against
$\kappa\,d\Lambda_t$ proves $\phi_t(x)-\phi_t(y)\le C$ and hence the
first inequality in \eqref{eq:lt-new-mobility-comparison}.  The same
kernel estimate and \eqref{eq:lt-new-source-exclusion} give
\[
 \phi_t(x)\le C+
 \frac{\kappa C_W}{2\pi}\log^+\frac{D_\Sigma}{c_1r},
\]
which proves the second inequality.  The mean-zero term is already
included in $K_q$; no exterior source contribution has been removed.
\end{proof}

Choose an isothermal coordinate map $\vartheta$ with
$h=e^{2\omega(x)}|dx|^2$.  On a bounded coordinate domain $V$, fix
\begin{equation}
 \Psi(\xi)=\vartheta^{-1}(z+r\xi),\qquad
 v(t)=f(t)\circ\Psi,\qquad
 \beta(t,\xi)=r^{-2}a(t,\Psi(\xi))e^{-2\omega(z+r\xi)}.
 \label{eq:lt-new-fixed-chart}
\end{equation}
We require $\Psi(\overline V)$ to stay in a compact subchart.  The
center and radius are held fixed in time.  The equation and its local
dissipation become exactly
\begin{equation}
 v_t=\beta\tau(v),\qquad
 \int_V\beta|\tau(v)|^2\,d\xi
 =\int_{\Psi(V)}a|\tau_h(f)|^2\,dA_h\le\mathscr D(t),
 \label{eq:lt-new-chart-equation}
\end{equation}
where $\tau$ and the energy measure $d\mu_v=\tfrac12|dv|^2d\xi$ are
Euclidean.  For $I=[s,t]$ set
\begin{equation*}
 B(u)=\sup_V\beta(u,\cdot),\qquad
 \Theta(I)=\int_I B(u)\,du,\qquad
 D(I)=\int_I\mathscr D(u)\,du.
\end{equation*}

\begin{lemma}
\label{lem:lt-new-transport}
For $\zeta\in C_c^\infty(V)$,
\begin{equation}
 \left|\int\zeta\,d\mu_{v(t)}-\int\zeta\,d\mu_{v(s)}\right|
 \le\|\zeta\|_\infty D(I)
 +\|d\zeta\|_\infty\sqrt{2E_0D(I)\Theta(I)}.
 \label{eq:lt-new-energy-transport}
\end{equation}
Moreover, for any fixed smooth map $Q:V\to N$ and
$\zeta\in C_c^\infty(V)$, define
$F_Q(u)=\tfrac12\int\zeta^2|d(v(u)-Q)|^2$, using a fixed isometric
embedding of $N$.  Then
\begin{align}
 \|v(t)-v(s)\|_{L^2(V)}^2&\le D(I)\Theta(I),
 \label{eq:lt-new-map-transport}\\
 |F_Q(t)-F_Q(s)|&\le
 \|\zeta\|_\infty^2D(I)+C_{Q,\zeta,E_0}\sqrt{D(I)\Theta(I)}.
 \label{eq:lt-new-profile-transport}
\end{align}
The last constant uses $Q$ only on the compact support of $\zeta$.
No derivative of $\beta$ is required.
\end{lemma}

\begin{proof}
The energy identity with a spatial test reads
\[
 \frac{d}{du}\int\zeta\,d\mu_v
 =-\int\zeta\beta|\tau(v)|^2
 -\int\langle d\zeta,\beta\tau(v)\otimes dv\rangle.
\]
Cauchy--Schwarz, \eqref{eq:lt-new-chart-equation} and
$E(v;V)\le E_0$ prove \eqref{eq:lt-new-energy-transport}.
For the map itself, Cauchy--Schwarz with the scalar weight $B(u)$ gives
\[
 \int_V\left|\int_s^t v_u\,du\right|^2
 \le\Theta(I)\int_I\int_V\frac{|v_u|^2}{B(u)}
 \le\Theta(I)\int_I\int_V\frac{|v_u|^2}{\beta}
 \le\Theta(I)D(I).
\]
Finally, since the normal component of $\Delta v$ is orthogonal to
$v_u\in T_vN$, integration by parts gives
\[
 \begin{split}
 F_Q'(u)={}&-\int\zeta^2\beta|\tau(v)|^2
 +\int\zeta^2\langle\Delta Q,v_u\rangle\\
 &-2\int\zeta\langle d\zeta\cdot d(v-Q),v_u\rangle.
 \end{split}
\]
The squared second Cauchy--Schwarz factors for the last two terms are
bounded, respectively, by
$\Theta(I)\int\zeta^2|\Delta Q|^2$ and
$C\Theta(I)\|d\zeta\|_\infty^2(E_0+E(Q;\operatorname{supp}\zeta))$.
This proves \eqref{eq:lt-new-profile-transport}.
\end{proof}

\subsection{Harmonic profiles at prescribed times}
\label{subsec:lt-common-profiles}

Let $\varepsilon_{\rm ell}>0$ be an energy threshold for the local
$L^2$-tension estimate: if $E(u;B_2)<\varepsilon_{\rm ell}$, then
\[
 \|u-u_{B_1}\|_{W^{2,2}(B_1)}
 \le C\bigl(\|du\|_{L^2(B_2)}+\|\tau(u)\|_{L^2(B_2)}\bigr).
\]
Here $N\subset\mathbb R^K$ is the fixed isometric embedding and
$u_B:=|B|^{-1}\int_B u\,d\xi$ is the Euclidean average; it need not
belong to $N$.
We use this estimate after fixed dilations of the disks; see
\cite[Lemma~2.1]{WangWeiZhang2017}.  Fix once and for all
\begin{equation*}
 \varepsilon_*:=\frac1{64}\min\{m_{\rm src},\varepsilon_{\rm ell}\}.
\end{equation*}

\begin{proposition}
\label{prop:lt-new-prescribed-compactness}
Let $t_j\to\infty$.  Consider fixed-in-time charts
$\Psi_j(\xi)=\vartheta^{-1}(z_j+r_j\xi)$, with uniformly controlled
background conformal factors on each testing domain, and write
$v_j(t)=f(t)\circ\Psi_j$.  Suppose on a domain $U$ that
\[
 \begin{aligned}
 v_j(t_j)&\rightharpoonup Q &&\text{in }H^1_{\rm loc}(U,\mathbb R^K),\\
 \mu_{v_j(t_j)}&\stackrel{*}{\rightharpoonup}\mu
                             &&\text{as Radon measures on }U.
 \end{aligned}
\]
Let $S=\{x\in U:\mu(\{x\})\ge\varepsilon_*\}$.  Then $S$ is
finite, and
\begin{equation*}
 v_j(t_j)\longrightarrow Q\quad\hbox{strongly in }
 H^1_{\rm loc}(U\setminus S),\qquad \tau(Q)=0\quad\hbox{on }U.
\end{equation*}
Here harmonicity on $U$ includes removal of the points in $S$.
Consequently
\begin{equation}
 \mu=\mu_Q+\sum_{x\in S}m_x\delta_x,\qquad m_x\ge\varepsilon_*.
 \label{eq:lt-new-atomic-measure}
\end{equation}
The same assertions hold for the root sequence $f(t_j)$ on $\Sigma$.
For charts exhausting $\mathbb R^2$, $Q$ has finite energy and extends
across infinity to a harmonic sphere, possibly constant.
\end{proposition}

\begin{proof}
Fix $K\Subset U_1\Subset V\Subset U\setminus S$.  In
\eqref{eq:lt-new-fixed-chart} use the
index $j$, and put
\[
 B_j(t)=\sup_V\beta_j(t,\cdot),\qquad
 \Theta_j(s)=\int_{t_j}^s B_j(t)\,dt.
\]
All subsequent estimates are first made on the set of times
$\Theta_j(s)\le1$.

\emph{Small-energy neighborhoods.}
There is a fixed $\delta>0$ and a finite family of smooth cutoffs,
with values in $[0,1]$,
supported in $V$ which equal one on neighborhoods covering
$\overline{U_1}$, and whose initial masses are less than
$3\varepsilon_*$.  They can be chosen so that every point of
$U_1$ has a coordinate $4\delta$-ball inside one of these
neighborhoods.  To see this, first choose small balls with limiting
mass less than $2\varepsilon_*$, using the absence of atoms of size
$\varepsilon_*$ on $\overline{U_1}$, and then use a finite cover and
weak convergence.  Shrink $\delta$ if necessary.  By
\eqref{eq:lt-new-energy-transport}, every one of these masses remains
less than $4\varepsilon_*$ while $\Theta_j\le1$, for all large $j$.
This is below both $m_{\rm src}$ and $\varepsilon_{\rm ell}$.

\emph{Mobility comparison and completion of the interval.}
For $x\in\Psi_j(U_1)$, an active packet cannot satisfy
$q+d_h(x,b_t(p,q))<c\delta r_j$, with $c>0$ sufficiently small
and independent of $j$.  Indeed,
\eqref{eq:lt-new-active-packet} puts its entire cutoff support
within distance $4q+d_h(x,b_t(p,q))$ of $x$, contradicting the
preceding mass bound.  Lemma~\ref{lem:lt-new-green-comparison}, and
the uniform upper and lower bounds for $e^{2\omega}$, therefore give
\begin{equation}
 \inf_{U_1}\beta_j(t,\cdot)\ge c_{U_1,V}B_j(t)
 \qquad(\Theta_j(t)\le1).
 \label{eq:lt-new-local-clock-comparison}
\end{equation}
For each fixed $j$ the same lemma also gives a strictly positive,
time-independent lower bound for $\beta_j$ on $U_1$ for these
times.  This lower bound may tend to zero with $j$.

There is a finite $s_j>t_j$ with $\Theta_j(s_j)=1$.  Otherwise the
small-energy and mobility estimates just obtained hold for all
$t\ge t_j$.  The positive lower bound for that fixed $j$ then
forces $\Theta_j(s)\to\infty$, a contradiction.  Thus the interval
exists without a uniform bound on its length.  Throughout it the
total dissipation is at most $\int_{t_j}^\infty\mathscr D(t)\,dt\to0$.

\emph{A comparison almost-harmonic slice.}
By \eqref{eq:lt-new-chart-equation} and
\eqref{eq:lt-new-local-clock-comparison},
\[
 \int_{t_j}^{s_j}B_j(t)
       \|\tau(v_j(t))\|_{L^2(U_1)}^2\,dt
       \le C\int_{t_j}^\infty\mathscr D(t)\,dt,
 \qquad \int_{t_j}^{s_j}B_j(t)\,dt=1.
\]
Choose $\widehat t_j\in[t_j,s_j]$ with
$\|\tau(v_j(\widehat t_j))\|_{L^2(U_1)}^2
\le2C\int_{t_j}^\infty\mathscr D(t)\,dt$.
The preserved small-energy bounds and the elliptic estimate on
smaller disks compactly contained in $U_1$ give strong $H^1$
compactness on compact subsets of $U_1$.
After a subsequence the limit $\widehat Q$ is smooth and harmonic
there.  Equation~\eqref{eq:lt-new-map-transport} gives
\[
 \|v_j(\widehat t_j)-v_j(t_j)\|_{L^2(U_1)}^2
 \le\int_{t_j}^\infty\mathscr D(t)\,dt\longrightarrow0,
\]
so $\widehat Q=Q$.

\emph{Return to the prescribed time.}
Choose $0\le\zeta\le1$ in $C_c^\infty(U_1)$, equal to one on $K$, and use this
smooth harmonic $Q$ in \eqref{eq:lt-new-profile-transport}.  It gives
\[
 \begin{split}
 \frac12\int_K|d(v_j(t_j)-Q)|^2
 &\le \frac12\int\zeta^2|d(v_j(\widehat t_j)-Q)|^2
       +\int_{t_j}^\infty\mathscr D(t)\,dt\\
 &\quad+C_{Q,\zeta,E_0}
          \left(\int_{t_j}^\infty\mathscr D(t)\,dt\right)^{1/2}
          \longrightarrow0.
 \end{split}
\]
Together with the $L^2$ convergence, this proves strong convergence
on $K$.  Apply the argument to a compact exhaustion; all local
limits agree with the original weak limit.  Equivalently, applying
it to any further subsequence shows strong convergence for the
chosen weakly convergent sequence.

Since $\#S\le E_0/\varepsilon_*$ and $Q$ has finite energy,
finite-energy removability gives harmonicity across $S$.
Strong convergence identifies the measure off $S$, proving
\eqref{eq:lt-new-atomic-measure}.  A finite isothermal atlas proves
the root statement by the same argument with fixed spatial scales.
For an expanding plane chart, finite energy also removes the point
at infinity.  The removability and sphere-gap facts used here are
recalled, for example, in \cite[Proposition~1.1]{Parker1996}.
\end{proof}

\subsection{PW packet charts and the vertex--edge decomposition}
\label{subsec:lt-new-extraction}

We construct the finite tree on one subsequence of the prescribed
times.  All its non-root charts will use the barycenters and
variances of accepted PW packets.

\medskip\noindent\emph{Representing a harmonic component by a PW packet.}
For an accepted packet $P=(p,q)$, choose an $h$-orthonormal frame
$\mathcal I_b:\mathbb R^2\to T_b\Sigma$ in the corresponding
neighborhood and define
\begin{equation}
 \Psi^{\rm PW}_{t,P}(\xi)
 :=\exp_{b_t(P)}\bigl(\sigma_t(P)\mathcal I_{b_t(P)}\xi\bigr).
 \label{eq:lt-pw-chart}
\end{equation}
When $\sigma_{t_j}(P_j)\to0$, these charts are injective on every
fixed disk for large $j$, with metric
$(\Psi^{\rm PW}_{t_j,P_j})^*g_{t_j}$.  

Suppose Proposition~\ref{prop:lt-new-prescribed-compactness} gives
a nonconstant harmonic profile $Q$ in an isothermal rescaling
$\vartheta^{-1}(c_j+\lambda_j\xi)$, with $\lambda_j\to0$.
Its limiting energy measure has the form
$\mu=\mu_Q+\sum m_x\delta_x$, with $\mu(\mathbb R^2)\le E_0$.
The harmonic-profile estimate of Section~4 remains valid
for this measure.  Indeed, its proof chooses $L=\max|dQ|>0$ and a pair
$(z,R)$ with $\tfrac12|dQ|^2\ge L^2/8$ on $B_R(z)$.  Since
$\chi_{z,R}=1$ on $B_R(z)$ and $\mu\ge\mu_Q$,
\begin{align*}
 M_\mu(z,R)&\ge \frac{\pi L^2R^2}{8}\ge m_{\rm sph}>\varepsilon_{\rm pkt}/2,\\
 \frac{\int\chi_{z,R}(y)|y-b|^2\,d\mu(y)}
      {M_\mu(z,R)R^2}
 &\ge \frac{\pi L^2R^2}{16E_0}\ge\rho_{\rm sph}^2>\rho^2_+.
\end{align*}
With $h=e^{2\omega}|dx|^2$, transfer this pair to
\[
 p_j=\vartheta^{-1}(c_j+\lambda_j z),\qquad
 q_j=e^{\omega(c_j)}\lambda_jR,\qquad P_j=(p_j,q_j).
\]
 Weak convergence of the energy measures therefore
gives convergence of the masses and uniform convergence, on bounded
sets of candidate centers, of the normalized Fr\'echet functionals.
Their Euclidean limit is strictly convex, so the barycenters and
variances converge as well.  Consequently
\begin{equation}
 W_{f(t_j)}(P_j)=1,\qquad
 \frac{\vartheta(b_{t_j}(P_j))-c_j}{\lambda_j}=O(1),\qquad
 0<c\le\frac{\sigma_{t_j}(P_j)}{\lambda_j}\le C<\infty.
 \label{eq:lt-pw-representation}
\end{equation}
After a subsequence, the transition from the PW chart
\eqref{eq:lt-pw-chart} to the original chart converges smoothly
on compact sets to a nondegenerate Euclidean similarity.  It
preserves strong $H^1$ convergence away from the transformed atom
set.  We henceforth represent this component in its PW chart and
denote the transformed harmonic map again by $Q$.

\medskip\noindent\emph{Completeness at concentration atoms.}
Recall the harmonic-sphere gap
\[
 \hbar_N=\inf\{E(Q):Q:S^2\to N
                   \text{ is nonconstant and harmonic}\}>0,
\]
with $\hbar_N=+\infty$ when the set is empty.  Let $M$ be a positive
atom of a root or rescaled limiting energy measure, placed at
zero.  Proposition~\ref{prop:lt-new-prescribed-compactness} gives
$M\ge\varepsilon_*$.  We show that it contains a subordinate
nonconstant profile of energy at most $M$.

For this completeness argument use an isothermal representative
of the chart, so that all subsequent rescalings are affine.
Write $u_j$ for the maps in these coordinates and
$\mu_j=\tfrac12|du_j|^2d\xi$.  Fix $\eta=\varepsilon_*/4$.
Choose $\delta_j\downarrow0$ and $k_j\to\infty$ sufficiently slowly
that
\begin{equation}
 \mu_j(B_{\delta_j})\longrightarrow M,\qquad
 \mu_j(B_{\delta_j/k_j})\longrightarrow M.
 \label{eq:lt-new-isolating-disks}
\end{equation}
This follows by diagonalization over fixed continuity radii of
the limiting measure.  For $0\le s\le\delta_j/4$ put
\[
 F_j(s)=\max_{c\in\overline B_{\delta_j/2}}\mu_j(B_s(c)).
\]
Smoothness makes $F_j$ continuous and nondecreasing, with
$F_j(0)=0$ and $F_j(\delta_j/k_j)>M-\eta$ for large $j$.
Let $\ell_j$ be the smallest radius with $F_j(\ell_j)=M-\eta$
and let $c_j$ be a maximizing center.  Then
$\ell_j\le\delta_j/k_j$.  By \eqref{eq:lt-new-isolating-disks},
the selected disk must meet $B_{\delta_j/k_j}$, since the energy
in the surrounding annulus tends to zero.  Thus
$|c_j|\le2\delta_j/k_j$.

Set $w_j(\xi)=u_j(c_j+\ell_j\xi)$.  Energy boundedness and
compactness of $N$ give a diagonal subsequence with
\[
 \begin{aligned}
 w_j&\rightharpoonup Q&&\text{in }H^1_{\rm loc}(\mathbb R^2),\quad
 \mu_{w_j}&\stackrel{*}{\rightharpoonup}\nu
                                      &&\text{as Radon measures}.
 \end{aligned}
\]
Every fixed rescaled disk lies inside $B_{\delta_j}$ for large $j$,
so $\nu(\mathbb R^2)\le M$.  For fixed $z\in\mathbb R^2$ and
$0<\rho<1$, the disk $B_{\rho\ell_j}(c_j+\ell_jz)$ is admissible
in the definition of $F_j$.  The Portmanteau inequality on open
balls, followed by $\rho\downarrow0$, therefore gives
\begin{equation}
 \nu(\{z\})\le M-\eta\qquad(z\in\mathbb R^2).
 \label{eq:lt-new-atom-decrease}
\end{equation}
On the other hand,
$\nu(\overline B_1)\ge M-\eta>0$, by the Portmanteau inequality
on this compact set.

The composed rescaling is still a translation and dilation in
the initial isothermal coordinate.  Applying Proposition
\ref{prop:lt-new-prescribed-compactness} to the displayed weak
limits shows that $Q$ is harmonic and that $\nu$ consists of
its energy and finitely many atoms.  If $Q$ is nonconstant, it
is the required sphere.  Otherwise the positive measure on
$\overline B_1$ contains an atom, to which the same construction
applies.  At each repeated step its mass decreases by at least
$\eta$, by \eqref{eq:lt-new-atom-decrease}, while every positive
atom has mass at least $\varepsilon_*$.  A nonconstant limit
must therefore occur after finitely many steps.  Its energy is
at most the original $M$, and its composed scale is $o(1)$
relative to the initial chart.

By \eqref{eq:lt-pw-representation}, this profile is represented
by an accepted PW packet at the same times $t_j$.  Thus every
positive atom contains a subordinate nonconstant PW profile
and has mass at least $\hbar_N$. 

\medskip\noindent\emph{The finite maximal family.}
For shrinking center--scale sequences
$\mathbf q=(x_j,r_j)$ and $\mathbf q'=(y_j,\lambda_j)$, put
\[
 \operatorname{Sep}_j(\mathbf q,\mathbf q')
 =\frac{r_j}{\lambda_j}+\frac{\lambda_j}{r_j}
                 +\frac{d_h(x_j,y_j)^2}{r_j\lambda_j}.
\]
After subsequences, charts are equivalent when this quantity
stays bounded and orthogonal when it tends to infinity.
Equivalent charts have limiting transitions given by Euclidean
similarities and represent the same harmonic component up to
such a change of coordinates. 

First take weak map and energy limits on the root.  Proposition
\ref{prop:lt-new-prescribed-compactness} gives a harmonic map
$Q_*:\Sigma\to N$ and strong convergence off its finite atom set.
Add a nonconstant profile whenever one exists, on a further
subsequence, in a chart orthogonal to the previously chosen ones;
represent it by an accepted PW packet using
\eqref{eq:lt-pw-representation}.  This replacement preserves
equivalence and orthogonality.  Pass to further subsequences
so that the relative scales and centers of each finite family
have limits in the extended sense.

To avoid counting the same energy more than once, we choose in the
root chart and in each selected PW chart compact cores obtained by
removing small neighborhoods of all child marks.  If a smaller chart
has a center that remains bounded in a parent chart, its physical core
lies in one of these removed neighborhoods.  Charts at comparable
scales have separated centers by orthogonality.  Hence, after passing
to a subsequence, the physical images of the chosen cores are pairwise disjoint.  Strong convergence gives
\begin{equation}
 E(Q_*)+\sum_{\alpha\in V_{\rm en}}E(Q_\alpha)\le E_\infty,
 \qquad \#V_{\rm en}\le E_0/\hbar_N
 \quad(\hbar_N<\infty).
 \label{eq:lt-new-profile-budget}
\end{equation}
Each addition costs at least $\hbar_N$, so the process stops
after finitely many steps.  At its stopping subsequence, no
chart orthogonal to the selected family, even on a further
subsequence, has a nonconstant harmonic limit.

Every selected profile is based at a root atom, since its
positive core energy collapses to its limiting root center.
Moreover, every atom of the root or of a selected component
has a selected energetic descendant: the completeness argument
finds a subordinate nonconstant profile, and maximality forces
it to be equivalent to a selected one.  If there are no
nonconstant harmonic spheres, the same argument excludes root
atoms and gives strong convergence on all of $\Sigma$.

\medskip\noindent\emph{Branching and constant packet charts.}
Use one fixed isothermal coordinate near each root atom.
Temporarily write an energetic PW chart in these coordinates
as a center $z_{\alpha,j}$ and a coordinate radius
\[
 r_{\alpha,j}
 =e^{-\omega(z_{\alpha,j})}\sigma_{t_j}(P_{\alpha,j}),
 \qquad z_{\alpha,j}=\vartheta(b_{t_j}(P_{\alpha,j})).
\]

Say that $\alpha$ contains $\beta$ when
$r_{\beta,j}/r_{\alpha,j}\to0$ and
$(z_{\beta,j}-z_{\alpha,j})/r_{\alpha,j}$ stays bounded.
Containing charts are totally ordered: comparable scales with
bounded relative centers would be equivalent; otherwise the
smaller chart is contained in the larger.  Each energetic vertex
therefore has a nearest energetic ancestor, or the root.
Taking limits of relative centers gives a preliminary tree.
Its only ambiguity is that several children may have the same
attachment mark.

For such a group of children, no child contains another.
Orthogonality then implies that every pairwise distance is
much larger than the larger child radius.  Let $R_j$ be the
diameter of their centers and take one member's center as
origin.  The scale $R_j$ is much larger than each child
radius and much smaller than the parent radius (and tends to
zero for a root parent).  The normalized centers lie in
$\overline B_1$ and have diameter one.  Insert a vertex at this
scale, and repeat for each proper subgroup with a common
normalized limit.  Subgroups strictly decrease in size, so
only finitely many vertices are inserted, each with at least
two children.  These inserted charts are orthogonal to every
energetic chart; an equivalent energetic chart would already
be an intervening ancestor in the preliminary tree.

Proposition~\ref{prop:lt-new-prescribed-compactness} also applies
to the inserted charts.  Their harmonic limits are constant
by maximality.  At any vertex, an atom not at a child mark
would, by completeness, produce an energetic descendant missing
from the containment tree.  Conversely, every child mark
carries an atom because an energetic descendant has a positive
compact core collapsing to that point.  Thus the exceptional
set in each vertex chart is precisely its child marks.

These ghost vertices can themselves be represented
by accepted PW packets.  Suppose that an inserted ghost vertex has child marks
$w_1,\dots,w_k\in\overline B_1$, with $k\ge2$ and diameter one.
Choose two marks, say $w_1,w_2$, with $|w_1-w_2|=1$.  The corresponding
child atoms have masses at least $m_{\mathrm{sph}}$.  The packet
centered at the origin with radius $2$ therefore has mass at least
$2m_{\mathrm{sph}}$, while its normalized variance is bounded below
by a positive constant depending only on $m_{\mathrm{sph}}$ and
$E_0$.  The strict threshold inequalities and convergence of the
packet moments imply that the corresponding physical packet is
accepted for all sufficiently large $j$.  Its barycenter and variance
are then used as the PW chart of the ghost vertex.

We have obtained a finite rooted tree
$\mathcal T=(V,\mathcal E)$ with
$V=\{*\}\sqcup V_{\rm en}\sqcup V_{\rm gh}$.
For every non-root vertex write
\[
 b_{\alpha,j}=b_{t_j}(P_{\alpha,j}),\qquad
 \sigma_{\alpha,j}=\sigma_{t_j}(P_{\alpha,j}),\qquad
 \Psi_{\alpha,j}=\Psi^{\rm PW}_{t_j,P_{\alpha,j}}.
\]
All these packets have weight one.  With frames chosen smoothly
in each branch, every edge with non-root parent satisfies
\begin{equation}
 \frac{\sigma_{\beta,j}}{\sigma_{\alpha,j}}\longrightarrow0,
 \qquad
 \Psi_{\alpha,j}^{-1}(b_{\beta,j})
       \longrightarrow w_{\beta\alpha}\in\mathbb R^2.
 \label{eq:lt-new-edge-data}
\end{equation}
Distinct children have distinct marks.  Root marks are the
limits of their barycenters on $\Sigma$.  On every vertex the
maps converge strongly in $H^1$ locally off its child marks,
to $Q_\alpha$, nonconstant at energetic vertices and constant at
ghost vertices.  The harmonic root may be constant or nonconstant.

\medskip\noindent\emph{The exact partition and the edge energies.}
Set $\sigma_{*,j}=1$ only as a convention for the root holes.
For integers $m\ge m_0$ define the following physical regions,
using the geodesic balls of $h$:
\begin{align*}
 C_{*,j}^m
   &=\Sigma\setminus\bigcup_{\beta\in\operatorname{ch}(*)}
                          \overline B_{1/m}(b_{\beta,j}),\\
 C_{\alpha,j}^m
   &=B_{m\sigma_{\alpha,j}}(b_{\alpha,j})
       \setminus\bigcup_{\beta\in\operatorname{ch}(\alpha)}
               \overline B_{\sigma_{\alpha,j}/m}(b_{\beta,j})
                       \qquad(\alpha\ne *),\\
 N_{(\alpha,\beta),j}^m
   &=B_{\sigma_{\alpha,j}/m}(b_{\beta,j})
                  \setminus\overline B_{m\sigma_{\beta,j}}(b_{\beta,j}).
\end{align*}
Thus the centers and radii in the partition are precisely PW
barycenters and variances; the integer $m$ only exhausts the
component interiors.

Choose $m_0$ so that the child marks at each vertex are separated
by more than $4/m_0$, all finite non-root marks lie in $B_{m_0-2}$,
and the root holes lie in disjoint normal neighborhoods.
For every fixed $m\ge m_0$ and large $j$, the child holes are
disjoint inside the parent outer ball.  Equation
\eqref{eq:lt-new-edge-data} also gives
$m\sigma_{\beta,j}<\sigma_{\alpha,j}/m$ for every edge.
Splitting each parent hole into its edge annulus and child outer
ball, and repeating down the finite tree, gives
\begin{equation}
 \Sigma=\coprod_{\alpha\in V}C_{\alpha,j}^m
          \coprod_{e\in\mathcal E}N_{e,j}^m
       \quad\text{up to boundaries of zero energy measure}.
 \label{eq:lt-new-partition}
\end{equation}
A diagonal subsequence makes these statements valid for
$m_0\le m\le j$.  The cores increase with $m$ and the edge
annuli decrease, wherever defined.

In a non-root PW chart,
$\sigma_{\alpha,j}^{-2}\Psi_{\alpha,j}^*h\to|d\xi|^2$
smoothly on compact sets.  Its moving core therefore tends to
\[
 K_\alpha^m
   =B_m\setminus\bigcup_{\beta\in\operatorname{ch}(\alpha)}
                         \overline B_{1/m}(w_{\beta\alpha}).
\]
On the root let $K_*^m$ be the complement of the geodesic
$1/m$-balls about its limiting child marks.  For fixed $m$,
these cores and their boundaries stay away from the marks.
Strong $H^1$ convergence and smooth convergence of the
background chart metrics give
\begin{equation}
 E_h(f(t_j);C_{\alpha,j}^m)\longrightarrow E(Q_\alpha;K_\alpha^m).
 \label{eq:lt-new-core-energies}
\end{equation}
Indeed, in dimension two the constant factor
$\sigma_{\alpha,j}^2$ cancels from the energy; the remaining
metric coefficients converge smoothly, the energy densities
converge in $L^1$, and the moving boundaries converge to
circles of zero limiting energy measure.  This reasoning
also proves the vanishing of ghost energy.

After a further diagonal subsequence,
$q_e(m)=\lim_{j\to\infty}E_h(f(t_j);N_{e,j}^m)$ exists for all
edges and integers $m\ge m_0$.  These numbers belong to $[0,E_0]$
and are nonincreasing in $m$.  Define
\begin{equation}
 \mathcal E_e^{\rm edge}
 :=\lim_{m\to\infty}\lim_{j\to\infty}E_h(f(t_j);N_{e,j}^m)
 =\inf_{m\ge m_0}q_e(m)\ge0.
 \label{eq:lt-new-edge-energy}
\end{equation}
For this fixed tree and subsequence, replacing the nested cutoffs
by mutually cofinal ones gives the same value, by the two
monotonicity inequalities.  No independence from a different
time subsequence or profile extraction is claimed.

For each fixed $m$, the partition \eqref{eq:lt-new-partition}
and convergence \eqref{eq:lt-new-core-energies} yield
\[
 E_\infty=\sum_{\alpha\in V}E(Q_\alpha;K_\alpha^m)
                          +\sum_{e\in\mathcal E}q_e(m).
\]
Letting $m\to\infty$ exhausts the smooth components away from
their finitely many marks and infinity.  The sums are finite,
the ghost energies are zero, and the edge terms decrease to
\eqref{eq:lt-new-edge-energy}.  Together with completeness and
\eqref{eq:lt-new-profile-budget}, this proves the following theorem.

\begin{theorem}[Long-time PW profiles and the vertex--edge decomposition]
\label{thm:lt-vertex-edge-decomposition}
Let $f$ be the global smooth GPHF solution with the fixed
construction and coupling of Sections~2--4.  Every sequence
$t_j\to\infty$ has a subsequence with a harmonic root map
$Q_*:\Sigma\to N$ and a finite rooted marked tree
$\mathcal T=(V,\mathcal E)$ such that:
\begin{enumerate}
 \item Every non-root vertex is represented by accepted packets
 $P_{\alpha,j}$ from the original PW construction, with
 $W_{f(t_j)}(P_{\alpha,j})=1$ and
 $\sigma_{t_j}(P_{\alpha,j})\to0$.  In their charts
 \eqref{eq:lt-pw-chart},
 \[
 f(t_j)\circ\Psi^{\rm PW}_{t_j,P_{\alpha,j}}
       \longrightarrow Q_\alpha
       \quad\text{strongly in }H^1_{\rm loc}
                       (\mathbb R^2\setminus Z_\alpha),
 \]
 where $Z_\alpha$ is the finite set of child marks.
 The limit is a nonconstant harmonic sphere for
 $\alpha\in V_{\rm en}$ and constant for
 $\alpha\in V_{\rm gh}$.  The root convergence is strong in
 $H^1$ locally away from its child marks.
 Each ghost has at least two children.
 \item The energetic packet charts are pairwise orthogonal and
 maximal: on no further subsequence does a shrinking chart
 orthogonal to all of them have a nonconstant harmonic limit.
 Every positive atom in any selected vertex chart or on the
 root has a selected energetic descendant and mass at least
 $\hbar_N$.  In particular,
 $\#V_{\rm en}\le E_0/\hbar_N$ if $\hbar_N<\infty$.
 \item For the edge energies \eqref{eq:lt-new-edge-energy},
 \begin{equation}
 E_\infty=E(Q_*)+
           \sum_{\alpha\in V_{\rm en}}E(Q_\alpha)
             +\sum_{e\in\mathcal E}\mathcal E_e^{\rm edge}.
 \label{eq:lt-vertex-edge-identity}
 \end{equation}
 If $N$ has no nonconstant harmonic sphere, the tree consists
 only of the root and $f(t_j)\to Q_*$ strongly in $H^1(\Sigma)$.
\end{enumerate}
Figure~\ref{fig:pw-bubble-tree-scheme} illustrates the scale--position
relations and the vertex and neck energies in a ``Mickey Mouse''
configuration, where a constant ghost component $\alpha$ connects
two nonconstant harmonic bubbles $\beta$ and $\gamma$.
\end{theorem}
\begin{figure}[H]
  \centering
  \includegraphics[width=\textwidth]{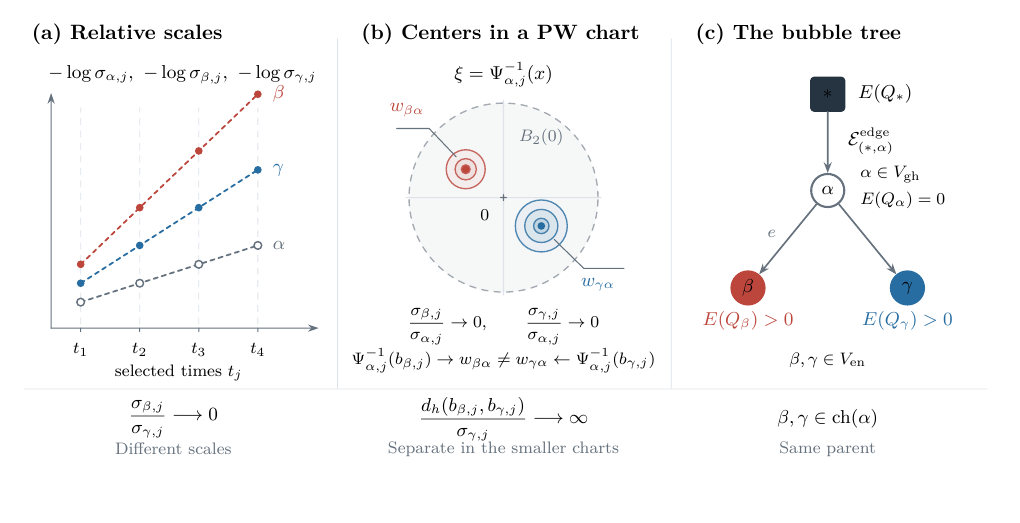}
  \caption{Scale and position determine the PW bubble tree.
  (a) Packet scales along a common selected sequence $t_j\to\infty$;
  dashed lines only connect the displayed samples.
  (b) In the chart $\Psi_{\alpha,j}$, the smaller packets shrink towards
  distinct child marks $w_{\beta\alpha}$ and $w_{\gamma\alpha}$.
  The translucent disks illustrate decreasing relative radii.
  (c) The corresponding branching component $\alpha\in V_{\rm gh}$
  is constant, while $\beta,\gamma\in V_{\rm en}$ are harmonic bubbles.
  Although $\sigma_{\beta,j}/\sigma_{\gamma,j}\to0$, their normalized
  separation diverges, so they are siblings. Vertex energies and
  neck energies give the energy identity displayed below the panels.}
  \label{fig:pw-bubble-tree-scheme}
\end{figure}

We conclude with a corotational disk example with $\gamma=2$. For the same initial and boundary data,
the ordinary harmonic map heat flow develops a finite-time singularity,
whereas the GPHF remains smooth and develops two separated harmonic
profiles at infinite time, with no-neck property. The proof is given in
the companion note \cite{GPHFCriticalRefinement}.

Let $D=B_1\subset\mathbb R^2$, and equip $N=S^2\times S^2$ with
\[
 h_N=h_{S^2_\alpha}+\mathfrak f(\alpha)h_{S^2_\beta},
 \qquad \mathfrak f(\alpha)=1+\epsilon_w g(\alpha).
\]
Here $g:[0,\pi]\to[0,1]$ is smooth, equals zero on $[0,\pi/2]$
and one on $[3\pi/4,\pi]$, and satisfies $g'>0$ on
$(\pi/2,3\pi/4)$. In the corotational class write
\[
 u(r,\theta,t)=(\alpha(r,t),\theta,\beta(r,t),\theta).
\]
Fix $b=\pi+\delta$, where $0<\delta<\pi/4$, with center and boundary
values
\[
 \alpha(0,t)=\beta(0,t)=0,
 \qquad \alpha(1,t)=\pi,\quad \beta(1,t)=b.
\]
Put
\[
 Q_\ell(r)=2\arctan(r/\ell),\qquad
 H_b(r)=\pi+2\arctan\!\bigl(r\tan(\delta/2)\bigr).
\]
Choose a smooth nondecreasing cutoff $\vartheta$ equal to zero on
$[0,1/2]$ and one on $[3/4,1]$. The initial map is
\[
 \begin{aligned}
  \alpha_0(r)&=(1-\vartheta(r))Q_{R_0}(r)+\vartheta(r)\pi,\\
  \beta_0(r)&=H_b(r)-(1-\vartheta(r))
                              \bigl(\pi-Q_{\lambda_0}(r)\bigr).
 \end{aligned}
\]
Thus the two transitions have initial scales
$0<\lambda_0\ll R_0\ll1$, and the map agrees with the stationary
root $(\pi,H_b)$ near the boundary.

Use the PW construction of Section~2 with the Dirichlet heat and
Green kernels on $D$, and denote the resulting source by
$S_u^{\mathrm{PW}}$. Fix the packet cutoffs to accept the standard
degree-one sphere $\mathcal Q(r,\theta)=(Q_1(r),\theta)$, and normalize
the source by a fixed multiplicative constant so that its total
weight for this isolated profile is $A_Q=2\pi$. Taking $\kappa=1$
then gives the critical calibration $\gamma_Q=2.$ We compare the original flow with its GPHF version: 
\[
 \partial_tu=\tau(u)
 \qquad\hbox{and}\qquad
 \begin{cases}
  \partial_tu=e^{-2\phi_u}\tau(u),\\
  -\Delta\phi_u=S_u^{\mathrm{PW}},\quad
                        \phi_u|_{\partial D}=0.
 \end{cases}
\]

The following
conclusions hold: 

\begin{enumerate}
\item The ordinary flow develops a singularity in finite time.
The critically normalized GPHF is smooth for every finite time.

\item For the GPHF there are smooth scales
$R(t)=e^{-S_\alpha(t)}$ and $\lambda(t)=e^{-S_\beta(t)}$ such that
\[
 R(t)\longrightarrow0,\qquad
 \frac{\lambda(t)}{R(t)}\longrightarrow0.
\]
Writing the angle functions in the coordinate $s=-\log r$ and setting
$Q(x)=2\arctan(e^{-x})$, one has
\[
 \begin{aligned}
  u(t)&\longrightarrow(\pi,H_b)
                      &&\text{on }D\setminus\{0\},\\
  (\alpha,\beta)(S_\alpha(t)+\cdot,t)
     &\longrightarrow(Q,\pi)
                      &&\text{on }\mathbb R,\\
  (\alpha,\beta)(S_\beta(t)+\cdot,t)
     &\longrightarrow(0,Q)
                      &&\text{on }\mathbb R.
 \end{aligned}
\]
These are full-time limits, locally in energy and uniformly on compact
sets. Both harmonic spheres have energy $4\pi$, and
\[
 \lim_{t\to\infty}E(u(t))
   =8\pi+\mathfrak f(\pi)E_b,
 \qquad E_b=2\pi(1+\cos b).
\]
The last term is the energy of the nonconstant root.  The convergence has the strong no-neck property: after exhausting
the root and bubble profiles, the remaining transition regions have
vanishing energy and image diameter uniformly in $t$. 

\item Along a sequence $t_j\to\infty$, the Green metric on the middle
neck is uniformly
asymptotic to the flat cylindrical metric.
\end{enumerate}
 The figure~\ref{fig:bubble_comparison} below is a schematic illustration of the comparison. The precise limits and normalization are given in\cite{GPHFCriticalRefinement}.

\begin{figure}[H]
  \centering
  \includegraphics[width=\textwidth]{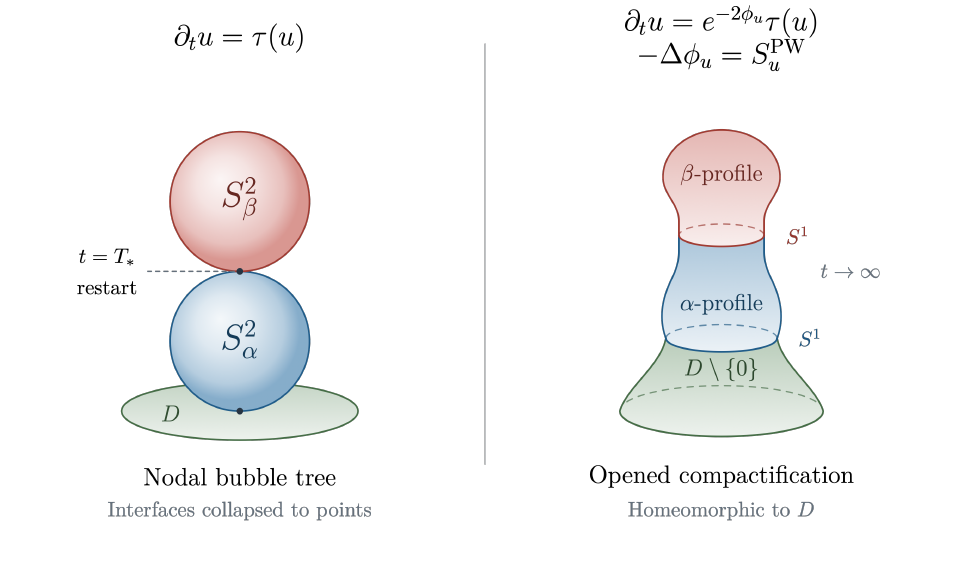}
  \caption{Keeping the circular interfaces between the root and the two profiles, and capping the terminal end, gives an opened domain homeomorphic to a disk. Collapsing those interfaces instead gives the usual nodal disk with two spherical components. }
  \label{fig:bubble_comparison}
\end{figure}

\backmatter

\section*{Statements and Declarations}

\subsection*{Funding}
The author received no specific funding for this work.

\subsection*{Competing interests}
The author declares no competing interests.

\subsection*{Data availability}
No datasets were generated or analysed in this theoretical study.

\subsection*{Use of artificial intelligence}
The author used ChatGPT and Codex (OpenAI) to assist with the development
and checking of mathematical arguments, manuscript drafting and revision, literature
searches, and the preparation of LaTeX and schematic figures. The author
takes responsibility for the mathematical arguments, references,
figures, and final manuscript.


\bibliography{GPHF_references}

\end{document}